\documentclass[11pt,draft,reqno]{amsart}

\usepackage{amssymb} 
\usepackage{mathrsfs} 
\usepackage{stmaryrd} 
\usepackage{color}

\DeclareMathAlphabet{\mathpzc}{OT1}{pzc}{m}{it} 

\newtheorem{Thm}{Theorem}[section]
\newtheorem{Lem}[Thm]{Lemma}
\newtheorem{Prop}[Thm]{Proposition}
\newtheorem{Def}[Thm]{Definition}

\theoremstyle{definition}
\newtheorem{Rem}[Thm]{Remark}

\theoremstyle{definition}

\theoremstyle{definition}
\newtheorem{Example}[Thm]{Example}

\theoremstyle{definition} 

\makeatletter
\@addtoreset{equation}{section}
\makeatother

\newcommand{\eps}{\varepsilon}

\newcommand\alg{\mathbb{A}} 

\newcommand\funzione{\longrightarrow}
\newcommand\ci{\mathbb{C}} 
\newcommand\erre{\mathbb{R}} 

\newcommand{\su}{\mathbb{S}}  

\renewcommand\j{\mathbf{j}}  
\renewcommand\k{\mathbf{k}} 

\newcommand{\quat}{\mathbb{H}}

\DeclareMathOperator{\re}{Re} 
\DeclareMathOperator{\im}{Im} 

\newcommand\setmeno{\!\smallsetminus\!} 

\providecommand{\clint}[1]{\hspace{0.045ex}[#1]} 
\providecommand{\opint}[1]{\hspace{0.15ex}\left]#1\right[\hspace{0.15ex}} 

\newcommand\norma[2]{\Vert #1\Vert_{#2}} 

\newcommand{\A}{\mathsf{A}}  
\newcommand{\B}{\mathsf{B}}  
\newcommand{\Lap}{\mathsf{L}}  
\newcommand{\Id}{\ \!\mathsf{Id}}

\newcommand\End{\textsl{End}}
\renewcommand\r{\textsl{r}} 

\newcommand\Lin{\mathscr{L}}

\newcommand\s{\mathpzc{s}} 

\newcommand{\C}{\mathsf{C}} 
\newcommand{\Q}{\mathsf{Q}}

\newcommand\zbar{\overline{z}}

\newcommand{\so}[1]{\mathcal#1} 

\newcommand{\T}{\mathsf{T}}  
\renewcommand{\S}{\mathsf{S}}  
\newcommand{\U}{\mathsf{U}}  
\newcommand{\F}{\mathsf{F}}  
\newcommand{\G}{\mathsf{G}}  
\newcommand{\R}{\mathsf{R}}  
\newcommand{\D}{\mathsf{D}}  

\newcommand{\1}{\mr{1}}

\providecommand{\opint}[1]{\hspace{0.15ex}\left]#1\right[\hspace{0.15ex}} 
\providecommand{\clsxint}[1]{\hspace{0.1ex}\left[#1\right[\hspace{0.15ex}} 
\DeclareMathOperator{\de}{d \! \hspace{0.2ex}} 

\renewcommand{\S}{\mathsf{S}}  
\newcommand{\V}{\mathsf{V}}  

\renewcommand\sp{\hspace{2.8ex}} 

\renewcommand\i{\mathbf{i}}

\newcommand\vbar{\overline{v}}

\newcommand\enne{\mathbb{N}} 

\newcommand\X{\textsl{X}\hspace{0.1ex}}

\newcommand\dualita[2]{\langle #1,#2 \rangle} 
\def\vuoto{\varnothing} 

\providecommand{\cldxint}[1]{\hspace{0.15ex}\left]#1\right]} 

\newcommand{\lra}{\longrightarrow} 
\newcommand{\RR}{\mathbb{R}}  
\newcommand{\CC}{\mathbb{C}}
\newcommand{\OO}{\Omega}
\newcommand{\mr}{\mathrm}

\newcommand{\mscr}{\mathscr}
\newcommand{\ray}{\mscr{R}}
\newcommand{\ce}{\mscr{C}}

\newcommand{\mc}{\mathcal}
\newcommand{\ui}{\mr{i}}
\newcommand{\II}{\mathfrak{I}}
\newcommand{\bs}{{\tiny$\blacksquare$}}

\definecolor{blu1}{rgb}{0.1,0.1,1}

\definecolor{verde}{rgb}{0.1,0.4,0.2}

\definecolor{pink}{rgb}{1.0, 0.33, 0.64}

\begin{document}


\title[Slice regular semigroups]{Slice regular semigroups}

\author{Riccardo Ghiloni and Vincenzo Recupero}

\address{\textbf{Riccardo Ghiloni}\\
        Dipartimento di Matematica\\ 
        Universit\`a di Trento\\
        Via Sommarive 14\\ 
        38123 Trento\\ 
        Italy. \newline
        {\rm E-mail address:}
        {\tt ghiloni@science.unitn.it}}
        
\address{\textbf{Vincenzo Recupero}\\
        Dipartimento di Scienze Matematiche\\ 
        Politecnico di Torino\\
        Corso Duca degli Abruzzi 24\\ 
        10129 Torino\\ 
        Italy. \newline
        {\rm E-mail address:}
        {\tt vincenzo.recupero@polito.it}}

\date{}


\begin{abstract}
In this paper we introduce the notion of slice regular right linear semigroup in a quaternionic Banach space. It is an operatorial function which is slice regular (a noncommutative counterpart of analyticity) and which satisfies a noncommutative semigroup law characterizing the exponential function in an infinite dimensional noncommutative setting. We prove that a right linear operator semigroup in a quaternionic Banach space is slice regular if and only if its generator is spherical sectorial. This result provides a connection between the slice regularity and the noncommutative semigroups theory, and characterizes those semigroups which can be represented by a noncommutative Cauchy integral formula. All our results are generalized to Banach two-sided modules having as a set of scalar any real associative *-algebra, Clifford algebras $\erre_n$ included.
\end{abstract}


\maketitle


\thispagestyle{empty}



\section{Introduction}\label{S:intro}

\subsection{The problem of analytic semigroups in the noncommutative setting}

A linear operators group, or more generally a linear operators semigroup on a real or complex Banach space $X$, is a mapping $\T: \clsxint{0,\infty} \funzione \Lin(X)$ such that $\T(0)$ is the identity and the deterministic law
\begin{equation}\label{classical semigroup law-intro}
 \T(t+s) = \T(t)\T(s) \qquad \forall t, s > 0
\end{equation}
is satisfied, $\Lin(X)$ being the space of bounded linear operators on $X$. For the general theory of operator semigroups we refer to \cite{EngNag00} and we recall here that, under the mild assumption that $y:= \T(\cdot)x$ is continuous for every $x \in X$, it is well-known that there exists the derivative $y'(0) =: \A x$ for every $x$ belonging to a dense subspace 
$D(\A)$ of $X$, and $y$ solves the Cauchy problem in $y'(t) = \A y(t)$, $y(0) = x \in D(\A)$. The linear operator 
$\A : D(\A) \funzione X$ is the so-called \emph{generator} of $\T$. If $\T$ is also continuous from $\clsxint{0,\infty}$ into 
$\Lin(X)$, then $\A$ turns out to be a bounded operator defined on the whole $X$ and 
$\T(t) = e^{t\A} := \sum_{n \ge 0}(t\A)^n/n!$, so that \eqref{classical semigroup law-intro} reads
\[
e^{(t+s)\A} = e^{t\A}e^{s\A}.
\]
Linear operators semigroups are a crucial tool for several topics in mathematics like partial differential equations, quantum mechanics, stochastic processes, control theory, and dyna\-mical networks; applications to other theoretical and applied sciences are also important, e.g. to open quantum systems, population dynamics, Boltzmann equations (cf. \cite[Chapter~VI]{EngNag00}).

The famous paper \cite{Sto32} of M.H. Stone \emph{``On one-parameter unitary groups in Hilbert spaces''} can be considered as the starting point of the modern theory of operator semigroups, whose development is witnessed by the fundamental monographs \cite{HilPhi57, Dav80, Paz83, Gol85, Lun95, Tai95, EngNag00} and by their references. Motivated by quantum mechanics (cf. \cite{Neu32b}), the paper of Stone, together with J. von Neumann's paper \cite{Neu32} \emph{``uber einen Satz von Herrn M.H. Stone''}, are a crucial step for the definition of the exponential map in infinite dimension.

G. Birkhoff and von Neumann in their celebrated paper \cite{BirNeu36} 
\emph{``The logic of quantum mechanics''} pointed out that quantum mechanics can be formulated not only in the nowaday classical setting of complex Hilbert spaces, but also on Hilbert spaces whose set of scalars is $\quat$, the skew-field of quaternions (cf. \cite{Sol95} for details). This remark originated the study of quantum mechanics in the quaternionic framework (see, e.g., \cite{FinJauSchSpe62, Emc63, HorBie84, CasTru85, Adl95}), whose natural set\-ting is a Hilbert two-sided $\quat$-module $X$, and where $\Lin(X)$ is replaced by the set $\Lin^\r(X)$ of bounded right linear operators acting on it (all the precise definitions will be recalled in Section~\ref{S:prelim}). However, the full development of the quaternionic formulation of quantum mechanics was prevented by the lack of a suitable quaternionic notion of spectrum (cf. \cite{CoGeSaSt,GhiMorPer13}). A~first rigorous formulation of quaternionic quantum mechanics has been started only in 2007 when 
the concept of \emph{spherical spectrum} of a quaternionic operator was introduced in \cite{CoGeSaSt07}. This new concept provides the basis for a proper application of the spectral theory to quaternionic quantum mechanics. Indeed, it permits to construct a noncommutative functional calculus for right linear operators on a Banach two-sided module over $\quat$ (and over a Clifford algebra as well, cf. \cite{ColSabStr08, ColSab09, ColSab10, CoGeSaSt, CoGeSaSt10, CoSaSt, GhiMorPer13}) and to deduce spectral representation theorems for normal operators in the quaternionic Hilbert setting (cf. \cite{ACK16,GhiMorPer-bis}).

The mentioned noncommutative functional calculus strongly relies on the theory of slice regu\-lar functions, recently introduced in \cite{GeSt}. Slice regular functions 
extend
to quaternions the classical concept of holomorphic function of a complex variable. They form a class of functions admitting a local power series expansion at every point of their domain of definition (cf. \cite{GenSto12}), including polynomials with quaternionic coefficients on one side.

In order to recall the notion of slice regular function let us first observe the fundamental fact that $\quat$ has a ``slice complex'' nature. This fact can be described as follows. If $\su \subseteq \quat$ is the set of square roots of $-1$ and if, for each $\j \in \su$, we denote by $\CC_\j$ the Euclidean plane of $\quat$ generated by $1$ and $\j$, then 
$\quat = \bigcup_{\j \in \su}\CC_\j$ and $\CC_\j \cap \CC_\k=\RR$ for every $\j,\k \in \su$ with $\j \neq \pm \k$. 
Therefore if $D$ is an open domain of $\CC$ invariant under complex conjugation and $\OO_D=\bigcup_{\j \in \su}D_\j$, where $D_\j:=\{r+s \, \j \in \CC_\j \, : \, r,s \in \RR, r+si \in D\}$, a function $f:\OO_D \lra \quat$ of class $C^1$ is called \emph{right slice regular} (resp. \emph{left slice regular}) if, for every $\j \in \su$, its restriction $f_\j$ to $D_\j$ is holomorphic with respect to the complex structures on $D_\j$ and on $\quat$ defined by the right (resp. left) multiplication by $\j$, i.e. if $\partial f_\j/\partial r+\partial f_\j/\partial s\ \! \j=0$ (resp. $\partial f_\j/\partial r+ \, \j \partial f_\j/\partial s=0$) on $D_\j$. This definition is naturally extended to functions with values in any Banach two-sided $\quat$-module, 
e.g.~$\Lin^\r(X)$. A remarkable property of slice regular functions is a Cauchy-type integral formula (cf. \cite{CoGeSa}). Let us consider first the left slice case. If $D$ is bounded with a piecewise $C^1$ boundary, and $f : \OO_D \funzione \quat$ is left slice regular and continuously extends on the closure of $\OO_D$ in $\quat$, then it holds:
\begin{equation}\label{intro:cauchy formula}
f(p)=\frac{1}{2\pi}\int_{\partial D_\j} C_q(p) \, \j^{-1} \de q \ f(q) \qquad \forall p \in \OO_D, \quad \forall \j \in \su,
\end{equation}
where $C_q(p)$ denotes the
\emph{$($left$\,)$ noncommutative Cauchy kernel}
\[
C_q(p) :=(p^2 - 2\re(q) p + |q|^2)^{-1}(\overline{q}-p),
\]
the line integral in \eqref{intro:cauchy formula} being defined in a natural way (see \eqref{def integrale}). The noncommutative Cauchy kernel $C_q$ is a left slice regular function, while for any fixed $p$ the function $q \longmapsto C_q(p)$ is right slice regular. The unusual fact that the differential $\de q$ appears on the left of $f(q)$ depends on the noncommutativity of $\quat$. If instead $f$ is right slice regular the noncommutative Cauchy integral formula reads 
$f(p)=\frac{1}{2\pi}\int_{\partial D_\j} f(q)  \, \j^{-1} \de q \ C^\r_q(p)$, where 
$C^\r_q(p) := (\overline{q}-p)(p^2 - 2\re(q) p + |q|^2)^{-1}$. 

As observed in \cite{CoGeSaSt,GhiMorPer13}, the classical notions of spectrum and of resolvent operator are not useful in order to define a noncommutative functional calculus. Cauchy integral formula \eqref{intro:cauchy formula} indicates a way to define new notions of spectrum and of resolvent ope\-rator, suitable for the noncommutative case: these notions are the \emph{spherical spectrum} and the \emph{spherical resolvent ope\-rator}. If $\A$ is a right linear operator on a Banach 
two-sided $\quat$-module $X$, then its \emph{spherical resolvent set} is the set of quaternions $q$ such that the operator
\[
\Delta_q(\A) := \A^2 - 2\re(q)\ \!\A + |q|^2 \Id
\]
is bijective and its inverse is bounded, where $\Id$ is the identity operator on $X$. Accordingly, the \emph{spherical spectrum} is the complement of the spherical resolvent set and the \emph{spherical resolvent operator} $\C_q(\A)$ is defined by
\[
  \C_q(\A) := \Delta_q(\A)^{-1}\overline{q} - \A\Delta_q(\A)^{-1}
\]
for every $q$ in the spherical resolvent set of $\A$.

The noncommutative functional calculus based on the spherical resolvent operator is exploited in \cite{ColSab11} in order to prove the counterpart of the classical generation theorems by Hille-Yosida and by Feller-Miyadera-Phillips for a strongly continuous right linear semigroup, i.e. a mapping $\T : \clsxint{0,\infty} \funzione \Lin^\r(X)$ such that $\T(\cdot)x$ is continuous for every $x \in X$. Their statements are analogous to the real and complex cases: the generator of $\A$ has the same formal definition and in particular we still have that $\A$ is bounded if and only if $\T$ is \emph{uniformly continuous}, i.e. $\T \in C(\clsxint{0,\infty};\Lin^\r(X))$; in this case $\T(t) = \sum_{n \ge 0}(t\A)^n/n!$. 

In paper \cite{GhiRec15} we show that the above-mentioned generation theorems for quaternionic right linear semigroups can be actually reduced to the classical commutative case by means of a simple technique, so that the functional calculus is not needed at this stage. In \cite{GhiRec15} we also introduce the class of spherical sectorial right linear operators and we prove that such operators generate a semigroup which can be represented by a Cauchy integral formula. Let us recall that a right linear operator $\A$ on $X$ is \emph{spherical sectorial} with vertex $\omega \in \erre$ if its spherical resolvent set contains a set of the form $\omega + \Omega_{\pi/2+\delta}$,  where
\[
  \Omega_{\pi/2+\delta} :=
   \{q \in \quat \setmeno \{0\} \, : \, \arg(q)<\pi/2+\delta\}
\]
for some $\delta \in \cldxint{0,\pi/2}$, with $\arg(q) := \theta \in \opint{0,\pi}$ if 
$q \in \quat \setmeno \RR$ and 
$q=r e^{\theta\j} \in \mathbb{C}_\j$, $\arg(q) := 0$ if $q \in \opint{0,\infty}$, and $\arg(q) := \pi$ if $q \in \opint{-\infty,0}$.
We prove that, if $\A$ has this property and satisfies the estimate 
\begin{equation}\label{sectorial estimate intro}
  \norma{\C_q(\A)}{} \le \frac{M}{|q - \omega|} \qquad \forall q \in \omega + \Omega_{\pi/2+\delta}
\end{equation}
for some $M \ge 0$, then the formula 
\begin{equation}\label{intro:e^At}
  \T(t) = \frac{1}{2\pi} \int_{\gamma_\j} \C_q(\A) \, \j^{-1}e^{tq} \de q \qquad \forall t > 0, 
\end{equation}
defines a strongly continuous right linear semigroup generated by $\A$, where $\j$ is an arbitrarily fixed element of $\su$ and $\gamma_\j$ is a suitable path of $\CC_\j$, surrounding the possibly unbounded spherical spectrum of $\A$
(in \cite{GhiRec15} we dealt with the case $\omega = 0$, the general case being proved in Theorem 
\ref{main thm AMS} below). As a consequence, the integral in \eqref{intro:e^At} is independent of $\j$ and the semigroup $\T(t)$ is analytic in time. Formula \eqref{intro:e^At} is clearly related to the Cauchy integral formula \eqref{intro:cauchy formula}, where the Cauchy kernel appears on the left: indeed the functions $q \longmapsto C_q(p)$ and $q \longmapsto C_q(\A)$ turn out to be both right slice regular. We underline that the noncommutative setting prevents from the possibility of applying the classical strategy for sectorial operators (see, e.g., \cite[Proposition 4.3, p. 97]{EngNag00}) and a different technique is needed (cf. \cite{GhiRec15}). We also point out a crucial difference between the scalar and operatorial quaternionic cases: if in 
\eqref{intro:cauchy formula} $pq = qp$ (i.e. when $p,q$ belong to the same $\CC_\j$), then $C_q(p)=(q-p)^{-1}$, and we find again the form of the classical Cauchy kernel for holomorphic functions, while in the operatorial case the commutation 
$\A q = q \A$ is in general false if $q$ is not real, so that the operatorial commutative and noncommutative cases are extremely different.

At this point there arises the problem to identify which kind of regularity characterizes the class of semigroups generated by spherical sectorial operators, in other terms we aim to find the class of right linear semigroups which can be represented by the noncommutative Cauchy integral formula \eqref{intro:e^At}. A major result in classical semigroups theory states that in the classical complex case this class is represented by the $\omega$-exponentially bounded \textit{analytic semigroups}, i.e. mappings $z \longmapsto \T(z)$ which are holomorphic in a sector $D_\delta \subseteq \ci$
with $\lim_{z \to 0} \T|_{\D_{\delta'}}(z)x = x$, $\sup_{z \in D_{\delta'}}\norma{\T(z)}{} e^{-\omega \re(z)} < \infty$ for every  subsector $D_{\delta'}$, $x \in X$, and satisfying the semigroup law 
\begin{equation} \label{eq:sl}
  \T(z + w) = \T(z)\T(w) 
\end{equation}
for $z, w \in D_\delta$. This result strongly connects the concept of operator semigroup to the theory of holomorphic functions (cf., e.g., \cite{Lun95,Tai95,EngNag00}).

The present paper is devoted to study this problem in the noncommutative case.


\subsection{A solution of the problem} 
If we first consider the simpler case of a bounded ope\-rator 
$\A \in \Lin^\r(X)$, then it turns out that the proper definition for $\T(q)$ is given by $\T(q) = \sum_{n\ge 0} (\A^n/n!)q^n$ since it uniquely extends $\T(t)$ in a right slice regular manner (in the analogous theory for left linear operators we would find $\sum_{n\ge 0} q^n(\A^n/n!)$). Anyway it turns out that $\T(p+q)$ is different from $\T(p)\T(q)$ even if $p$ and $q$ commute, and this occurs for any other ``reasonable'' extensions of $\T(t)$, i.e. $\sum_{n\ge 0} q^n(\A^n/n!)$, 
$\sum_n (\A q)^n/n!$, $\sum_n (q \A)^n/n!$.

In order to understand what is the point here and to find the proper semigroup law in the noncommutative framework, let us consider again the concept of slice regularity with values in $\quat$, or in $\Lin^\r(X)$, or generally in a Banach two-sided 
$\quat$-algebra $Y$, i.e. the natural noncommutative quaternionic counterpart of a Banach algebra (cf. Definition 
\ref{H-banach algebra} below). One fundamental observation is that the pointwise product of two right slice regular functions is not a right slice regular function. The proper notion of product is instead given by the \emph{slice product}, which can be easily illustrated for polynomial functions or power series. Indeed if we consider for instance series with coefficients in $Y$ on the left of the indeterminate $q$, then it is well-known that the proper way to perform the multiplication consists in imposing commutativity of $q$ with the coefficients (cf. \cite{Lam91}). Thus if $f(q) = \sum_n a_n q^n$ and 
$g(q) = \sum_n b_n q^n$, then their Cauchy product (or convolution) is defined by
\begin{equation}\label{intro:p*q}
  (f*g)(q) := \sum_{n}\bigg(\sum_{k+h=n} a_k b_h\bigg) q^n.
\end{equation}
Note that this product is different from the pointwise product of $f$ and $g$. This happens even when one of the two polynomials is constant, indeed if $g(q) = b_0$ the pointwise product is $f(q)g(q) = \sum_n a_n q^n b_0$, while 
$(f*g)(q) = \sum_n a_n b_0 q^n$. The general notion of \emph{slice product} between two right slice regular functions $f$ and $g$, which is given in Definition \ref{D:slice product} below and will be denoted simply by $f \cdot g$, turns out to be the natural generalization to functions of the product \eqref{intro:p*q} of power series. Since we are particularly interested in operator-valued functions (e.g. $\T(t) = \sum_n (\A^n/n!)q^n$ if $\A$ is bounded), let us consider the case 
$Y = \Lin^\r(X)$ where the product is the composition of operators. If $\F : \Omega_D \funzione \Lin^\r(X)$ and 
$\G : \Omega_D \funzione \Lin^\r(X)$ are two given right slice regular operatorial functions, then the function 
$q \longmapsto \F(q)\G(q)$ is not right slice regular in general, and the correct notion of product turns out to be the slice product $\F \cdot \G$, that in the special operatorial case $Y = \Lin^\r(X)$ will be denoted by the symbol $\F \odot \G$. For simplicity let us consider again the case of power series: if $(\A_n)$ and $(\B_n)$ are two sequences in $\Lin^\r(X)$ and if 
$\F(q) = \sum_n \A_n q^n$ and $\G(q) = \sum_n \B_n q^n$, then
we have
\[
  (\F \odot \G)(q) := \sum_{n}\bigg(\sum_{k+h=n} \A_k \B_h\bigg) q^n.
\]
We are now in position to describe the main result of our paper. We prove that if $\A$ is a spherical sectorial operator with vertex $\omega$ satisfying \eqref{sectorial estimate intro}, then it generates an $\omega$-exponentially bounded 
\emph{right slice regular semigroup}, i.e. a mapping 
$\T : \Omega_{\delta} \cup \{0\} \funzione \Lin^\r(X)$ such that $\T|_{\Omega_\delta}$ is right slice regular, 
$\lim_{q \to 0} \T|_{\Omega_{\delta'}}(q)x = x$, $\sup_{z \in \Omega_{\delta'}} \norma{\T(q)}{} e^{-\omega \re(q)} < \infty$ for every 
$q \in \Omega_{\delta'}$, $\delta' \in \opint{0,\delta}$, $x \in X$, and the following \textit{noncommutative right linear operator semigroup law} holds
\begin{equation} \label{eq:nsl}
  \T(p+q) = \T(p) \odot_p \T(q) \qquad \text{$\forall p,q \in \Omega_\delta$ with $pq = qp$},
\end{equation}
where $\T(p) \odot_p \T(q)$ means that we are considering the slice product with respect to $p$, with $q$ fixed. Vice versa we prove that if $\T$ is an $\omega$-exponentially bounded right slice regular semigroup, then its generator is spherical sectorial with vertex $\omega$. Thus we have obtained the following theorem.

\vspace{.7em}

\noindent
\textbf{Theorem H.} \textit{Let $X$ be a Banach two-sided $\quat$-module, let $\T:\clsxint{0,\infty} \funzione \Lin^\r(X)$ be a strongly continuous right linear semigroup and let $\A:D(\A) \funzione X$ be its right linear generator. Then $\A$ is a spherical sectorial operator with vertex $\omega$ satisfying \eqref{sectorial estimate intro} if and only if $\T$ extends to an 
$\omega$-exponentially bounded right slice regular semigroup.}

\vspace{.7em}

This theorem implies that right slice regular semigroups provide the class of semigroups which can be represented by Cauchy integral formula \eqref{intro:e^At}, namely the infinite dimensional exponential in a noncommutative framework.

Theorem H is a particular case of our main result,  Theorem \ref{thm:main}, 
which is valid in a very general noncommutative setting when the set of scalars $\quat$ is replaced by an arbitrary associative real *-algebra $\alg$, including e.g. all the Clifford algebras $\erre_n$. Indeed the relevant subset of this kind of algebras is the so-called \textit{quadratic cone} 
$Q_\alg$, which enjoys the same slice complex nature of $\quat$, i.e. $Q_\alg = \bigcup_{\j \in \su_{\alg}}\CC_\j$, where 
$\su_{\alg} := \{q \in \alg\ :\ q^2 = -1,\ q^c = -q\}$, $q \longmapsto q^c$ being the operation of *-involution (conjugation), and 
$\ci_\j$ denotes again the Euclidean plane of $\alg$ generated by $1$ and $\j$. This fact allows to employ many arguments of the quaternionic case, even if additional difficulties may arise, due mainly to the existence of zero-divisors. A central point of this analysis is the introduction of the general definition of a slice regular function with values in a Banach two-sided 
$\alg$-module. This new notion requires the concept of \emph{vector stem function} (see \cite{GhiPer11} for the scalar case) and unifies all the different notions of slice regular function disseminated in the literature 
(cf. \cite{GeSt, GhiPer11, ColSab14, ACKS-in-press}). The passage to the vector framework
introduces a difficulty which is not present in the classical commutative complex case, since when we evaluate a right slice regular operator-valued function
$q \longmapsto \F(q)$ at a vector $x$, we obtain that $q \longmapsto \F(q)x$ is not right slice regular anymore (cf. Example 
\ref{eq:no-right-slice} below).
This difficulty is evident in handling the noncommutative counterpart of the Laplace transform (see Section
\ref{SS:laplace transform}), an important tool for the proof of Theorem H.

We point out that our results comprise the classical ones as a particular case. Indeed, if $\alg=\ci$ and $X$ is a usual complex Banach space in which $zx=xz$ for $x \in X$ and $z \in \ci$, then \eqref{eq:nsl} reduces to \eqref{eq:sl} and \eqref{intro:e^At} coincides with the standard Cauchy integral formula for analytic semigroups, because 
$\C_z(\A)=(z\Id-\A)^{-1}$.

\subsection{Structure of the paper} 
The next section is devoted to some prelimi\-nary notions and properties concerning real *-algebras $\alg$. In Section \ref{S:right linear operators} we recall the precise definition of Banach two-sided $\alg$-module, we introduce the natural notion of Banach two-sided $\alg$-algebra and we describe an important example of this kind of algebras, the one of right linear operators acting on a Banach two-sided $\alg$-module. In Section \ref{S:vector slice functions} we define the general concept of slice regular function with values in a Banach two sided $\alg$-module and we prove its main properties, while in the following Section \ref{S:slice regular examples} we provide a list of relevant examples, including right power series, noncommutative exponentials, slice compositions of operatorial functions and spherical resolvent operators. In 
Section~\ref{semigroups} we recall the definition of right linear operator semigroups and we introduce the new class of right slice regular semigroups. The last section is devoted to prove that right slice regular semigroups represent precisely the class of semigroups generated by a spherical sectorial operator.
  

\section{Preliminaries}\label{S:prelim}

Let us assume that
\begin{equation}\label{A real algebra}
  \text{$\alg$ is a nontrivial real algebra with unit $1_\alg$},
\end{equation}  
i.e. we are given a real vector space $\alg \neq \{0\}$, endowed with a bilinear product 
$\alg \times \alg \funzione \alg: (p,q) \longmapsto pq$ whose unit is $1_\alg$. The simplest examples are provided by the set of real numbers $\erre$ and by the complex plane $\ci$, but in general we will admit that the product in $\alg$ is noncommutative, as in the case of the skew-field $\quat$ of quaternions, whose precise definition will be recalled in  Example \ref{examples of A} below. From the bilinearity of the product it follows that
\begin{equation}\label{r(ab) = (ra)b = a(rb)}
  r(pq) = (rp)q = p(rq) \qquad \forall r \in \erre, \quad \forall p, q \in \alg. 
\end{equation}
In this way we can identify the algebra of real numbers $\erre$ with the subalgebra of $\alg$ generated by $1_\alg$, thus $1=1_\alg$ and the notation $rq$ is not ambiguous if $r \in \erre$ and $q \in \alg$. Notice that 
\begin{equation}\label{ra = ar}
  rq = qr \qquad \forall r \in \erre, \quad \forall q \in \alg. 
\end{equation}
We can therefore consider the following well-known generalization of the complex conjugation.
\begin{Def}\label{D:*-algebra}
Assume that \eqref{A real algebra} holds. We say that a mapping $\alg \funzione \alg : q \longmapsto q^c$ is a 
\emph{*-involution} if it is $\erre$-linear and
\begin{alignat}{3}
  & (q^c)^c = q 		& \qquad &  \forall q \in \alg,  \notag \\  
  & (pq)^c = q^c p^c 	& \qquad & \forall p, q \in \alg,  \notag \\ 
  & r^c = r 			& \qquad & \forall r \in \erre.  \notag
\end{alignat}
If $\alg$ is endowed with a *-involution, we also say that $\alg$ is a \emph{real *-algebra}. 
\end{Def}

In the remainder of the paper we will assume that $\alg$ is associative and its real dimension is finite. We will summarize this and the previous assumptions by saying that
\begin{equation}\label{assumption on A}
  \text{$\alg$ is a finite dimensional associative nontrivial real *-algebra with unit,}
\end{equation}
and we will endow $\alg$ with the (Euclidean) topology induced by any norm on it.

\begin{Def}
Assume that \eqref{assumption on A} holds. The \emph{imaginary sphere in $\alg$} is defined by
\begin{equation}\label{imaginary sphere}
  \su_\alg := \left\{q \in \alg\ :\ q^c = -q,\ q^2 = -1\right\}, 
\end{equation}
and we set
\begin{equation}
  \ci_\j := \left\{r+s\j \in \alg\ :\ r, s \in \erre\right\}, \qquad 
  \j \in \su_\alg, \notag
\end{equation}
i.e. $\ci_\j$ is the real vector subspace of $\alg$ generated by $1$ and $\j \in \su_\alg$ 
or, equivalently, the real subalgebra of $\alg$ generated by $\j$. The \emph{quadratic cone} $Q_\alg$ is defined by
\begin{equation}\label{quadratic cone}
  Q_\alg := \bigcup_{\j \in \su_\alg} \ci_\j. 
\end{equation}
Finally the \emph{real part $\re(q)$} and the \emph{imaginary part $\im(q)$} of an element $q \in \alg$ are defined by
\begin{equation}\label{real and imaginary parts}
  \re(q) := (q+q^c)/2, \quad \im(q) := (q-q^c)/2,
  \qquad 
  q \in \alg.  
\end{equation}
\end{Def}

Observe that $Q_\alg$ is a real cone and that every $q \in Q_\alg$ satisfies the real quadratic equation $q^2 - 2\re(q)q + qq^c = 0$, which justifies the name ``quadratic cone''. In general, $Q_\alg$ is not a real vector subspace of $\alg$ (cf. Remark \ref{rem:commutativity} below).

In general $\re(q)$ and $\im(q)$ are not real numbers, at variance with the customary complex notation. If $z \in \ci$ then we set $\Re(z) := (z + \overline{z})/2 \in \erre$ and $\Im(z) := (z - \overline{z})/2i \in \erre$.

In the remainder of the paper, except for Section \ref{S:right linear operators}, we will assume that
\begin{equation}\label{S_A nonvuota}
  \su_\alg \neq \varnothing,
\end{equation}
in particular this removes from consideration the set of real numbers $\erre$. For the reader's convenience, in the following proposition, we give the proof of some useful properties enjoyed by a real *-algebra.

\begin{Prop}\label{useful properties}
Assume that \eqref{assumption on A} and \eqref{S_A nonvuota} hold. Then 
\begin{itemize}
\item[(a)] 
  For every $\j \in \su_\alg$ we have that $1$ and $\j$ are linearly independent and
  \begin{align*}
   & r, s \in \erre, \quad q = r + s \j \quad \Longrightarrow\ \quad q^c = r - s \j, \quad qq^c = q^cq = r^2 + s^2,\\
   & p,q \in \ci_\j \quad \Longrightarrow \quad pq=qp.
  \end{align*}
\item[(b)]
The following properties hold:
\begin{alignat}{3}
  & q^n \in Q_\alg \qquad \forall q \in Q_\alg, \quad \forall n \in \enne,  \notag \\
  & q \in Q_\alg \setmeno \{0\} \quad \Longrightarrow\ \quad \exists q^{-1} = (qq^c)^{-1} q^c \in Q_{\alg} \setmeno \{0\}.   
     \label{inverse of x} \notag
 \end{alignat}
 \item[(c)]
We have that
 \begin{equation}\label{eq:intersection}
  \ci_\j \cap \ci_\k = \erre \qquad \forall \j, \k \in \mathbb{S}_\alg, \ \j \neq \pm \k.
\end{equation}
\item[(d)]
The following set equalities hold:
\begin{gather}
   \su_\alg = \{q \in Q_\alg\ :\ q^2 = -1\}, \notag \\
   Q_\alg = \erre \cup \{q \in \alg \ :\ \re(q) \in \erre,\ qq^c \in \erre,\ qq^c > \re(q)^2\}. \notag
\end{gather}
In particular if $q \in Q_\alg \setmeno \erre$, then $\im(q)\im(q)^c > 0$,
$\j := \im(q)/\sqrt{\im (q) \im (q)^c} \in \su_\alg$, and $q = \re(q) + \sqrt{\im(q) \im (q)^c} \,\j \in \ci_\j$.
\item[(e)]
$Q_\alg=\alg$ if and only if $\alg$ is (a real $^*$-algebra) isomorphic to $\ci$ or $\quat$. 
In this case, if $p, q \in \alg=Q_\alg$, then
\begin{equation} \label{eq:commutative}
  pq = qp \quad \Longleftrightarrow \quad \exists \k \in \su_\alg\ :\ p, q \in \ci_\k.
\end{equation}
\end{itemize}
\end{Prop}

\begin{proof}
(a)
If $r, s \in \erre$, $s \neq 0$, and $r+s\j = 0$, then $(r/s)^2 = (-\j)^2 = -1$, a contradiction leading to the linear independence of $1$ and $\j$. The properties of the *-involution yields, for $r, s \in \erre$, $(r+s\j)^c = r + s \j^c = r - s \j$. The formula for $qq^c$ and the equality $pq=qp$ are easily verified.

(b)
The two properties follow from an easy induction and a direct computation.

(c)
If $(\ci_\j \cap \ci_\k) \setmeno \erre \neq \vuoto$ then there are $r, s \in \erre$, $s \neq 0$,  such that $\k = r + s\j$, therefore the equality $-1 = \k^2 = r^2 - s^2 + 2rs\j$ yields $r = 0$ and $s^2 = 1$. It follows that $\k = \pm \j$.

(d)
If $q \in Q_\alg$ and $q^2 = -1$, then there are $r , s \in \erre$, $\j \in \su_\alg$ such that $q = r + s\j$, $r=0$, and 
$s^2 = 1$. Therefore $q^c = -q$ and the characterization for $\su_\alg$ is proved. Concerning the second equality, from (a) it follows that, for every $q \in Q_\alg \setmeno \erre$, we have $\re(q) \in \erre$, $qq^c \in \erre$ and 
$qq^c > \re(q)^2$. On the other hand if $q \in \alg \setmeno \erre$ satisfies these three conditions, then $\im(q) \neq 0$ (otherwise $q = (q+q)/2 = (q+q^c)/2 \in \erre$) and $qq^c = q(q+q^c) - q^2 =  (q+q^c)q - q^2 = q^c q$, therefore 
$\im(q)\im(q)^c = qq^c - \re(q)^2 > 0$, $\j := \im(q)/\sqrt{\im(q)\im(q)^c} \in \su_\alg$, 
$q = \re(q) + \sqrt{\im(q)\im(q)^c} \, \j \in \ci_\j$ and (d) is proved.

(e) By the second part of (b), if $\alg=Q_\alg$, then $\alg$ is a division algebra and hence Frobenius' theorem implies that 
$\alg$ is isomorphic to $\ci$ or $\quat$ (cf. \cite[\textsection\ 8.2.4]{numbers}). The converse implication and \eqref{eq:commutative} are evident if $\alg=\ci$ and well-known if $\alg=\quat$ (cf. Example \ref{examples of A} and Remark \ref{rem:commutativity}).
\end{proof}

\begin{Rem}
Part (d) of Proposition \ref{useful properties} shows that definitions \eqref{imaginary sphere} and \eqref{quadratic cone} are consistent with the apparently different definitions given in \cite{GhiPer11}. \bs
\end{Rem}

\begin{Def}\label{circular sets}
Assume that \eqref{assumption on A} and \eqref{S_A nonvuota} hold. If $\j \in \su_\alg$, we define the real $^*$-algebra isomorphism $\phi_\j : \ci \funzione \ci_\j$ by setting
\begin{equation}\label{phiJ C-isomorphism}
 \phi_\j(r + si) := r + s\j, \qquad 
 r,s \in \erre.  \notag
\end{equation}
Given a subset $D$ of $\ci$, invariant under complex conjugation, the \emph{circular set associated to $D$} is the subset $\Omega_D$ of $Q_\alg$ defined by
\begin{equation}\label{Omega D}
  \Omega_D := \bigcup_{\j \in \su_\alg} \phi_\j(D) 
                     = \{r + s \j \in Q_\alg \, : \, r, s \in \erre, \, r+s i \in D, \, \j \in \su_\alg\}.  \notag
\end{equation}
A subset of $Q_\alg$ is said to be \emph{circular} if it is equal to $\Omega_D$ for some set $D$ as above. 
\end{Def}

Observe that if $D$ is open in $\CC$, then $\Omega_D$ is a relatively open subset of $Q_\alg$, because the function 
$Q_\mathbb{A} \to \ci : q \longmapsto \re(q)+i\sqrt{\im(q)\im(q)^c}$ easily extends to a continuous function on the whole $\alg$.

We recall that a real algebra $\alg$ satisfying \eqref{assumption on A} is said to be \textit{Banach} if it is equipped with a (complete) norm $|\cdot|$ which is submultiplicative, i.e. $|pq| \leq |p||q|$ for every $p,q \in \alg$, and $|1|=1$.

In what follows, we will often assume that 
\begin{equation}\label{eq:assumption-2}
  \alg \text{ is Banach with a norm $|\cdot|$ such that, for every $\j \in \su_\alg$, $|pq| = |p||q|$ if $p, q \in \ci_\j$.}\end{equation}
Observe that \eqref{eq:assumption-2} implies the compactness of $\su_\alg$. Indeed, by definition \eqref{imaginary sphere}, $\su_\alg$ is closed in $\alg$. Moreover $\su_\alg$ is contained in the compact sphere $\{q \in \alg \ :\ |q|=1\}$, because 
$|q|^2=|q^2|=|-1|=1$ if $q \in \su_\alg$. As an immediate consequence of the compactness of $\su_\alg$, one obtains that $Q_\alg$ is closed in $\alg$. We remark that \eqref{eq:assumption-2} is ensured by the following condition
\begin{equation}\label{eq:assumption-2bis}
 \alg \text{ is Banach with a norm } |\cdot| \text{ such that $|q|^2=qq^c$ for every $q \in Q_\alg$.}
\end{equation}
Notice that under assumption \eqref{eq:assumption-2bis} $\phi_\j$ is an isometry. It is worth also observing that \eqref{eq:assumption-2} and \eqref{eq:assumption-2bis} are equivalent if the norm $|\cdot|$ is induced by a scalar product on $\alg$ (cf. \cite[\textsection 10.1]{numbers}).

\begin{Example}\label{examples of A}
A remarkable class of associative real $^*$-algebras is the one of Clifford algebras (cf. \cite{GM91,GHS08} and 
\cite[Section 1]{GhPeSt}). Let $p,q \in \enne$, let $n=p+q$ and let $\mc{P}(n)$ be the family of all subsets of $\{1,\ldots,n\}$, where $\mc{P}(0)=\vuoto$. Identify $\RR$ with the vector subspace $\RR \times \{0\}$ of 
$\RR^{2^n}=\RR \times \RR^{2^n-1}$ and denote by $\{e_K\}_{K \in \mc{P}(n)}$ the canonical basis of $\RR^{2^n}$, where $e_{\vuoto}:=1$. For convenience, indicate $e_{\{k\}}$ also by $e_k$ if $k \in \{1,\ldots,n\}$. Let us define a real bilinear and associative product on $\erre^{2^n}$ by imposing that
\begin{itemize}
 \item $1$ is the neutral element;
 \item $e_k^2=1$ if $k \in \{1,\ldots,p\}$ and $e_k^2=-1$ if $k \in \{p+1,\ldots,n\}$;
 \item $e_ke_h=-e_he_k$ if $k,h \in \{1,\ldots,n\}$ with $k \neq h$;
 \item $e_K=e_{k_1}\cdots e_{k_s}$ if $K \in \mc{P}(n) \setmeno \{\vuoto\}$ and $K=\{k_1,\ldots,k_s\}$ with 
 $k_1<\ldots<k_s$.
\end{itemize}
This product on $\erre^{2^n}$ defines the so-called \emph{Clifford algebra $\mathit{C}\ell_{p,q}$ of signature $(p,q)$}, which is denoted also by $\erre_{p,q}$. Evidently, such an associative real algebra is not commutative if $n \geq 2$. The \textit{Clifford conjugation} of $\erre_{p,q}$ is the $^*$-involution $x \longmapsto \overline{x}$ which fixes $e_K$ if $K$ has $s$ elements and $s \equiv 0,3 \ \mr{mod}\ 4$ and sends $e_K$ into $-e_K$ if $s \equiv 1,2 \ \mr{mod}\ 4$. Endowing 
$\erre_{p,q}$ with Clifford conjugation, we obtain a real $^*$-algebra satisfying \eqref{assumption on A}. However, such an algebra $\erre_{p,q}$ does not have both properties \eqref{S_A nonvuota} and \eqref{eq:assumption-2} if $p \geq 1$: 
\begin{itemize}
 \item $\su_{\erre_{0,0}}=\vuoto$ ($\erre_{0,0}=\erre$ indeed) and $\su_{\erre_{1,0}}=\vuoto$, so $\erre_{0,0}$ and 
 $\erre_{1,0}$ do not verify \eqref{S_A nonvuota}.
 \item $\su_{\erre_{2,0}}$ and $\su_{\erre_{1,1}}$ are 2-hyperboloids in $\erre^4$ (recall that $\erre_{2,0}$ and $\erre_{1,1}$ are isomorphic) and hence they are not compact. It follows that $\erre_{p,q}$ does not admit any norm with property \eqref{eq:assumption-2} if $p \geq 2$ or $p=1$ and $q \geq 1$, because in these cases 
 $\su_{\erre_{2,0}} \subset \su_{\erre_{p,q}}$ or $\su_{\erre_{1,1}} \subset \su_{\erre_{p,q}}$.
\end{itemize}

Let us consider the case $p=0$ and $n=q \geq 1$. For simplicity, we use the alternative notation~$\erre_n$ instead of 
$\erre_{0,n}$. By direct inspection, one verifies that a point $x=\sum_{K \in \mc{P}(n)}x_Ke_K$ of $\erre_n$ with 
$x_K \in \erre$ belongs to the quadratic cone $Q_{\erre_n}$ of $\erre_n$ if and only if it satisfies the following polynomial equations  
\[
x_K=0 \quad \mbox{and} \quad \langle x , xe_K \rangle=0 \quad \text{for every $K \in \mc{P}(n) \setmeno \{\vuoto\}$ with $e_K^2=1$},
\]
where $\langle \cdot , \cdot \rangle$ denotes the standard scalar product on $\RR_n=\RR^{2^n}$. On $\erre_n$ it is defined the following submultiplicative norm, called \textit{Clifford operator norm}:
\[
|x|_{\mathit{C}\ell}:=\sup\{|xa| \in \RR \, : \, |a|=1\},
\]
where $|\cdot|$ indicates the Euclidean norm of $\erre_n=\erre^{2^n}$. It turns out that:
\begin{itemize}
 \item $Q_{\erre_n}=\erre_n$ if and only if $n \in \{1,2\}$. In particular, $\erre_1$ and $\erre_2$ are division algebras.
 \item $|x|_{\mathit{C}\ell}=|x|=\sqrt{x\overline{x}}$ for every $x \in Q_{\erre_n}$ and hence $|\cdot|_{\mathit{C}\ell}=|\cdot|$ if $n \in \{1,2\}$. If $n \geq 3$, 
 the Euclidean norm $|\cdot|$ of $\erre_n$ is not submultiplicative (e.g. 
 $|(1+e_{\{1,2,3\}})^2|=\sqrt{8}>2=|1+e_{\{1,2,3\}}|^2$) and $\erre_n$ has zero divisors (e.g 
 $(1+e_{\{1,2,3\}})(1-e_{\{1,2,3\}})=0$).
\end{itemize}

Endowing $\erre_n$ ($n \geq 1$) with Clifford conjugation and Clifford operator norm, we obtain a Banach real $^*$-algebra satisfying \eqref{S_A nonvuota} and \eqref{eq:assumption-2bis}. \textit{In what follows we always consider $\erre_n$ equipped with such a structure of Banach real $^*$-algebra}. The cases $n=1$ and $n=2$ are very important:
\begin{itemize}
 \item $\erre_1$ coincides with $\ci$ endowed with the standard conjugation, if we set $e_1=i$.
 \item $\erre_2$ is called \textit{algebra of quaternions}. Usually it is denoted by $\quat$ and one writes $i$, $j$ and $k$ in place of $e_1$, $e_2$ and $e_{\{1,2\}}$, respectively. \bs
\end{itemize}
\end{Example}

\begin{Rem} \label{rem:commutativity}
Two quaternions $p,q \in \quat$ commute if and only if they belong to the same slice $\ci_\j$. Let 
$p,q \in \quat \setmeno \erre$ and let $\j,\k \in \su_\quat$ such that $p \in \ci_\j$ and $q \in \ci_\k$. The equality $pq = qp$ is equivalent to $\j\k = \k\j$. Since $\j\k - \k\j=(\j - \k)(\j + \k)$ and $\quat$ has no zero divisors, we conclude that $p$ and $q$ commute if and only if $\j=\pm \k$, i.e. $p$ and $q$ belong to the same slice $\ci_\j$. This is not true in $\erre_n$ if 
$n \geq 3$; indeed, $e_3,e_{\{1,2\}} \in \su_{\erre_3}$, $e_3 \neq \pm e_{\{1,2\}}$, but $e_3e_{\{1,2\}}=e_{\{1,2\}}e_3$. The reader observes that $Q_{\erre_n}$ is not a real vector subspace of $\erre_n$ if $n \geq 3$. Indeed, since $e_3$ and $e_{\{1,2\}}$ commute, $e_3+e_{\{1,2\}}$ does not belong to $Q_{\erre_3}$, because $(e_3+e_{\{1,2\}})(\overline{e_3+e_{\{1,2\}}}\,)=2+2e_{\{1,2,3\}} \not\in \erre$. \bs 
\end{Rem}


\section{Two-sided $\alg$-algebras}\label{S:right linear operators}

 

\subsection{Two-sided modules and algebras}\label{SS:bimodules}

Let us recall that, if $\alg$ satisfies \eqref{assumption on A}, an abelian group $(X,+)$ is a \emph{left $\alg$-module} if it is endowed with a left scalar multiplication $\alg \times X \funzione X : (q,x) \longmapsto qx$ such that
\begin{alignat}{5}
  & q(x+y) = q x + q y 	& \qquad 	& \forall x, y \in X,	& \quad	& \forall q \in \alg, \notag \\
  & (p+q)x = p x + q x 	& \qquad 	& \forall x \in X, 		& \quad 	& \forall p, q \in \alg, \notag \\   
  & 1x = x 			& \qquad 	& \forall x \in X, \notag \\  
  & p(qx) = (pq)x 		& \qquad 	& \forall x \in X, 		& \quad 	& \forall p, q \in \alg. \notag
\end{alignat}
An abelian subgroup $Y$ of $X$ is a \emph{left $\alg$-submodule} if $qx \in Y$ whenever $x \in Y$ and $q \in \alg$. If 
$\alg$ is a field we obtain the classical notions of (left) vector space and subspace. 

The definition of \emph{right $\alg$-module} is completely analogous: it is required that the abelian group $(X,+)$ is endowed with a right scalar multiplication $X \times \alg \funzione X : (x,q) \longmapsto xq$ such that
\begin{alignat}{5}
  & (x+y)q = xq + yq & \qquad	& \forall x, y \in X,	& \quad	& \forall q \in \alg, \notag \\
  & x(p+q) = xp + xq & \qquad	& \forall x \in X, 		& \quad 	& \forall p, q \in \alg, \notag \\   
  & x 1= x	& \qquad 	& \forall x \in X, \notag \\  
  & (xp)q = x(pq) 		& \qquad	& \forall x \in X, 		& \quad 	& \forall p, q \in \alg. \notag
\end{alignat}
An abelian subgroup $Y$ of $X$ is a \emph{right $\alg$-submodule} if $xq \in Y$ whenever $x \in Y$ and $q \in \alg$.

\begin{Def}
Assume that \eqref{assumption on A} holds and let $(X,+)$ be an abelian group. We say that $X$ is a 
\emph{two-sided $\alg$-module} (or \emph{$\alg$-bimodule}) if it is endowed with two scalar multiplications 
$\alg \times X \funzione X : (q,x) \longmapsto q x$ and $X \times \alg \funzione X : (x,q) \longmapsto xq$ such that $X$ is both a left $\alg$-module and a right $\alg$-module and 
\begin{alignat}{3}
  & p(xq) = (p x)q	& \qquad 	& \forall x \in X, \quad \forall p, q \in \alg,   \notag \\
  & r x = x r 		& \qquad 	& \forall x \in X, \quad \forall r \in \erre. \label{prod per reali}
\end{alignat}
An abelian subgroup $Y$ of $X$ is a \emph{two-sided $\alg$-submodule} if it is both a left and a right $\alg$-submodule of $X$.
\end{Def}

If $\alg$ were simply a ring, then \eqref{r(ab) = (ra)b = a(rb)} and \eqref{ra = ar} make no sense, thus condition 
\eqref{prod per reali} should be omitted (see, e.g., \cite[Chapter 1, Section 2, p. 26-28]{AndFul74}). In our case 
$\alg$ is an algebra and it is natural to require \eqref{prod per reali}.

In \cite{AndFul74} it is suggested a self-explanatory notation which is useful when we consider different sets of scalars simultaneously: if $X$ is an abelian group then
\begin{align}
  & \text{$_\alg X$ means that $X$ is considered as a left $\alg$-module}, \notag \\
  & \text{$X_\alg$ means that $X$ is considered as a right $\alg$-module}. \notag
\end{align}

\begin{Def}\label{D:norma su V}
Assume \eqref{assumption on A} and \eqref{eq:assumption-2} hold and let $X$ be a two-sided $\alg$-module. A function 
$\norma{\cdot}{} : X  \funzione \clsxint{0,\infty}$ is called an \emph{$\alg$-norm on $X$} if 
\begin{alignat}{3}
  & \norma{x}{} = 0 \ \Longleftrightarrow \ x = 0, \notag \\
  & \norma{x + y}{} \le \norma{x}{} + \norma{y}{} 		& \qquad & \forall x, y \in X,  \notag \\
  & \norma{q x}{} \le |q| \, \norma{x}{}, \quad  \norma{x q}{} \le |q| \, \norma{x}{} 	
  & \qquad & \forall x \in X, \quad \forall q \in \alg. \label{eq:homog} 
\end{alignat}
Equipped with this kind of norm, $X$ is called a \emph{normed two-sided $\alg$-module} and we endow it with the topology induced by the metric $d : X \times X \funzione \clsxint{0,\infty} : (x,y) \longmapsto \norma{x-y}{}$. Finally, we say that $X$ is a \emph{Banach two-sided $\alg$-module} if this metric $d$ is complete.
\end{Def}

Observe that if $q \in Q_\alg \setmeno \, \{0\}$ and $x \in X$, then \eqref{eq:assumption-2} implies that
$\norma{xq}{} \leq \norma{x}{}|q| = \norma{x q q^{-1}}{} |q| \leq \norma{xq}{} |q^{-1}||q| = \norma{xq}{}$, therefore
$\norma{x q}{} = \norma{x}{} |q|$. A similar argument applies to $\norma{q x}{}$, therefore we have the following
result.

\begin{Lem}\label{norma moltiplicativa}
Assume \eqref{assumption on A} and \eqref{eq:assumption-2} hold, and let $X$ be a normed two-sided $\alg$-module.
Then
\[
\norma{qx}{} = \norma{xq}{} = |q|\norma{x}{} \qquad \forall x \in X, \quad \forall q \in Q_\alg.
\]
\end{Lem}

\begin{Rem}
If $X$ is a normed two-sided $\alg$-module whose $\alg$-norm is $\norma{\cdot}{}$, then, since $\RR \subseteq Q_\alg$, the preceding lemma implies that $\norma{\cdot}{}$ is a norm on $_\RR X$ in the usual real sense. Therefore the metric on $X$ is the one induced by $\norma{\cdot}{}$ as a standard norm on $_{\erre}X$. Finally observe that $X$ is a Banach 
two-sided $\alg$-module if and only if $_\erre X$ is a real Banach space. \bs
\end{Rem}

\begin{Def}\label{H-banach algebra}
Assume that \eqref{assumption on A} holds. A two-sided $\alg$-module $X$ is called \emph{(associative) two-sided 
$\alg$-algebra} if it is endowed with an associative product $X \times X \funzione X : (x,y) \longmapsto xy$ such that 
\begin{alignat}{4}
  & x(y + z) = xy + xz & \qquad 	& \forall x, y, z \in X, 	&  \notag \\
  & (x+y)z = xz + yz  	& \qquad 	& \forall x, y, z \in X, 	& \notag \\
  & q(xy) = (qx)y 	& \qquad 	& \forall x, y \in X, 	& \quad \forall q \in \alg, \notag \\
  & (xy)q = x(yq) 	& \qquad 	& \forall x, y \in X, 	& \quad \forall q \in \alg. \label{(xy)q = x(yq)} \notag
\end{alignat}
If we also assume that \eqref{eq:assumption-2} holds, then we say that $X$ is a 
\emph{normed two-sided $\alg$-algebra} provided $X$ is endowed with an $\alg$-norm $\norma{\cdot}{}$ such that 
$\norma{xy}{} \le \norma{x}{} \norma{y}{}$ for every $x, y \in X$. If $X$ is complete we say that $X$ is a 
\emph{Banach two-sided $\alg$-algebra}. If in addition $X$ is nontrivial and has a unit $1_X$ such that 
$\norma{1_X}{} = 1$, then $X$ is called a \emph{Banach two-sided $\alg$-algebra with unit}.
\end{Def}

\begin{Example} \label{exa:b-t-s-a}
Assume $\alg$ satisfies \eqref{assumption on A} and \eqref{eq:assumption-2}, e.g. $\alg=\erre_n$. Given a nonempty set $S$, the set of bounded $\alg$-valued functions on $S$, equipped with the pointwise operations of sum, of product, of left and right multiplications by scalars in $\alg$, of $^*$-involution $f^c(s):=(f(s))^c$ and endowed with supremum norm 
$\|f\|_\infty:=\sup_{s \in S}|f(s)|$, is a Banach two-sided $\alg$-algebra with unit. In particular, this is true for each power 
$\alg^m$. If $S$ has a topological structure, then the same pointwise defined operations make the set of bounded continuous $\alg$-valued functions on $S$ a Banach two-sided $\alg$-algebra with unit. \bs
\end{Example}

Another example of Banach two-sided $\alg$-algebra with unit, which is of crucial importance in this paper, is the one of right linear operators on a Banach two-sided $\alg$-module. We present this example in Section 
\ref{S:vector slice functions}.


\subsection{Right linear operators}

Let us recall the concept of right linear operators acting on a two-sided $\alg$-module. Assume that
\[
\text{\emph{$\alg$ satisfies \eqref{assumption on A} and $X$ is a Banach two-sided $\alg$-module.}}
\]
\begin{Def}
Let $D(\A)$ be a right $\alg$-submodule of $X$. We say that $\A : D(\A) \funzione X$ is \emph{right linear} if it is additive and
\begin{equation*}
 \A(xq) = \A(x) q \qquad \forall x \in D(\A), \quad \forall q \in \alg.
\end{equation*}
As usual, the notation $\A x$ is often used in place of $\A(x)$. We use the symbol $\End^{\r}(X)$ to denote the set of right linear operators $\A$ with $D(\A) = X$. The identity operator is right linear and is denoted by $\Id_X$ or simply by $\Id$ if no confusion may arise. Moreover, if $X$ is a normed two-sided $\alg$-module, then we say that 
$\A : D(\A) \funzione X$ is \emph{closed} if its graph is closed in $X \times X$. As in the classical theory, we set 
$D(\A^n) := \{x \in D(\A^{n-1}) \, : \, \A^{n-1} x \in D(\A)\}$ for every $n \in \enne \setmeno \{0\}$. 
\end{Def}

Let us also recall the following definition (see, e.g., \cite[Chapter 1, p. 55-57]{AndFul74}).

\begin{Def}
Let $D(\A)$ be a right $\alg$-submodule of $X$ and let $q \in \alg$. If $\A : D(\A) \funzione X$ is a right linear operator, then we define the mapping $q \A : D(\A) \funzione X$ by setting
\begin{equation}
   (q \A)(x) := q \A(x), \qquad 
   x \in D(\A).  \label{operatore per scalare a sx}
\end{equation}  
If $D(\A)$ is also a left $\alg$-submodule of $X$, then we can define $\A q :D(\A) \funzione X$ by setting
\begin{equation}
(\A q)(x):=\A(q x), \qquad 
x \in D(\A). \label{operatore per scalare a dx}
\end{equation} 
The sum of operators is defined in the usual way.
\end{Def}
It is easy to see that the operators defined in \eqref{operatore per scalare a sx} and \eqref{operatore per scalare a dx} are right linear.

\begin{Def}
Assume $X$ is normed with $\alg$-norm $\norma{\cdot}{}$. For every $\A \in \End^\r(X)$, we set
\begin{equation}\label{norma operatoriale}
  \norma{\A}{} := \sup_{x \neq 0} \frac{\norma{\A x}{}}{\norma{x}{}}
\end{equation}
and we define the set
\[
   \Lin^\r(X) := \{\A \in \End^\r(X)\ :\ \norma{\A}{} < \infty\}. \notag \\
\]
\end{Def}

Observe that $\norma{\A}{}$ can be equivalently defined as the operatorial norm of $\A$ as an element of 
$\End(_\erre X)$, therefore
\begin{align}\label{altra descrizione di Lr(V)}
  \Lin^\r(X) 
    & = \{\A \in \End(_\erre X) \, : \, \A \text{ is right linear}, \ \norma{A}{} < \infty\} \notag \\
    & = \{\A \in \Lin(_\erre X) \, : \, \A \text{ is right linear}\}, \notag
\end{align}
where $\Lin(_\erre X) = \{\A \in \End(_\erre X)\ :\ \norma{\A}{} < \infty\}$ is the usual normed $\RR$-vector space of continuous $\erre$-linear operators on $_\RR X$. The sum of operators, the scalar multiplications 
\eqref{operatore per scalare a sx} and \eqref{operatore per scalare a dx}, the composition, and \eqref{norma operatoriale}, make $\Lin^\r(X)$ a normed two-sided $\alg$-algebra with unit $\Id$. If $X$ is Banach, then $\Lin^\r(X)$ is Banach. Let us recall the following lemma (cf. \cite[Lemma 2.19]{GhiRec15}).

\begin{Lem}\label{L:L^r chiuso in L}
Let $X$ be a normed two-sided $\alg$-module with $\alg$-norm $\norma{\cdot}{}$. The $\erre$-vector subspace 
$\Lin^\r(X)$ of $\Lin(_{\erre}X)$ is closed with respect to the topology of pointwise convergence and hence with respect to the uniform operator topology of $\Lin(_{\erre}X)$.
\end{Lem}

It is also useful to consider the following complex structures on the two-sided $\alg$-module $X$.

\begin{Def}\label{D: complex space X_j}
Assume \eqref{S_A nonvuota} holds and let $\j \in \su_\alg$. We endow the abelian group $(X,+)$ with the complex scalar multiplication $\ci \times X \funzione X$ defined~by
\begin{equation}\label{product defining X_j}
  z x := x\phi_\j(z), \qquad 
  x \in X, \ 
  z \in \ci.
\end{equation}
The resulting complex vector space will be denoted by $X_\j$. If $\A : D(\A) \funzione X$ is a right linear operator, then we define the complex subspace $D(\A_\j)$ of $X_\j$ and the $\ci$-linear operator $\A_\j : D(\A_\j) \funzione X_\j$ by setting $D(\A_\j) := D(\A)$ and $\A_\j(x) := \A(x)$ for every $x \in D(\A_\j)$.
\end{Def}

\begin{Rem}
(i) Fix $\j \in \su_\alg$. Since $\ci_\j \subseteq Q_\alg$, if $X$ is normed with $\alg$-norm $\norma{\cdot}{}$, then Lemma \ref{norma moltiplicativa} ensures that $\norma{\cdot}{}$ is a norm on $X_\j$ in the usual complex sense. It is immediate to verify that $(X, \norma{\cdot}{})$ is a Banach two-sided $\alg$-module if and only if $(X_\j, \norma{\cdot}{})$ is a complex Banach space.

(ii) Let $\j \in \su_\alg$ and let $\|\cdot\|$ be an $\alg$-norm on $X$. Denote by $\Lin(X_\j)$ the $\ci$-vector space of continuous $\ci$-linear operators defined on the whole $X_\j$, equipped with the usual pointwise operations of sum and scalar multiplication. We have that $\Lin^\r(X) \subseteq \Lin(X_\j) \subseteq \Lin(_\erre X)$, the second inclusion being strict if $X \neq \{0\}$. 
If $X \neq \{0\}$ and there exists $q \in \mathbb{A}$ such that $\j q - q \j$ is invertible in $\mathbb{A}$ (this is true if, e.g., 
$\alg = \quat$), then the operator $X \to X: x \longmapsto x\j$ belongs to $\Lin(X_\j) \setmeno \Lin^\r(X)$ and the first inclusion is strict too. Furthermore, if $\alg$ coincides with the real subalgebra generated by $Q_\alg$ (e.g. if $\alg$ is equal to some 
$\erre_n$), then $\Lin^\r(X) = \bigcap_{\i \in \su_\alg} \Lin(X_\i)$.

(iii) There would be no need to introduce the notation $\A_\j$, the notion of mapping being a set-theoretical one. Anyway this is convenient to shorten some statements about $\A$ considered as a linear operator on a complex vector space. \bs
\end{Rem}


\section{Slice functions with values in a two-sided $\alg$-module}\label{S:vector slice functions}

The aim of this section is to introduce the notion of vector-valued slice regular function and to study its properties. We assume that
\[
  \text{\emph{$\alg$ is a real algebra satisying \eqref{assumption on A}, \eqref{S_A nonvuota} and \eqref{eq:assumption-2}.}}
\]
and that
\begin{equation}
  \text{\emph{$X$ is a Banach two-sided $\alg$-module with $\alg$-norm $\|\cdot\|$}}. \notag
\end{equation}
In order to introduce the notion of $X$-valued slice function, we consider $X$ as a real vector space, i.e. $_\erre X$, and we define in $X \times X$ a structure of complex vector space by defining the standard componentwise sum 
and the scalar multiplication $\ci \times (X \times X) \funzione (X \times X): (z,v) \longmapsto zv$:
\begin{equation}\label{product by a complex scalar in A_C}
 (r + si)(x,y) := (rx - sy, ry + sx)
\end{equation}
for $z = r + s i$, $v = (x, y)$, $r, s \in \erre$, $x, y \in X$. Endowing $X \times X$ with this complex vector space structure, we obtain the so-called complexification $X \otimes_\erre \ci$ of $X$. The \emph{complex conjugation} of 
$v = (x,y) \in X \otimes_\erre \ci$ is defined by 
\begin{equation}
  \vbar := (x,-y). \notag
\end{equation}  
We make $X \otimes_\erre \ci$ a real Banach space by defining $\norma{(x,y)}{} := \max\{\norma{x}{}, \norma{y}{}\}$ for 
$(x,y) \in X \otimes_\erre \ci$, thus if $D$ is a nonempty open subset of $\ci \simeq \erre^2$, then 
$C^1(D;X \otimes_\erre \ci)$ will denote the set of real continuously differentiable functions from $D$ into 
$X \otimes_\erre \ci$ in the sense of differential calculus in real Banach spaces. If in addition 
\begin{equation}
  \text{$X$ is a Banach two-sided $\alg$-algebra with unit $1_X$,} \notag
\end{equation}
the following product makes $X \otimes_\erre \ci$ a complex algebra: 
\begin{equation}\label{product in A_C}
  (x,y)(x',y'):=(xx'-yy',xy'+yx').  \notag
\end{equation}
By setting $\1 := (1_X, 0) \in X \otimes_\erre \ci$ and $\ui := (0,1_X) \in X \otimes_\erre \ci$, every 
$v = (x,y) \in X \otimes_\erre \ci$ can be uniquely written in the form $v = x1 + y\ui = x + y \ui$, and $\ui$ is called an \emph{imaginary unit}: $X \otimes _\erre \ci = X + X \ui = \{v = x+y\ui\ :\ x, y \in X\}$ and $\ui^2= -1$. Observe that 
$iv = \ui v = v \ui$ for every $v \in X \otimes_\erre\ci$ and the structure of real vector space induced by the scalar multiplication \eqref{product by a complex scalar in A_C} with $s = 0$ is the same of $_\erre X \times\ \!\! _\erre X$.

\begin{Rem}
Let $D$ be a nonempty open subset of $\ci$. Using \cite[Theorem 3.31, p. 79]{Rud73}, the vector Cauchy integral formula and standard complex analysis arguments for scalar functions, it is easy to check that the following statements are equivalent:
\begin{itemize}
\item[(i)] 
  $F \in C^{1}(D;X \otimes_\erre \ci)$ and $\frac{\partial F}{\partial r}+ i \frac{\partial F}{\partial s}  = 0$.
\item[(ii)]
  $F$ is complex differentiable in $D$.
\item[(iii)]
  $z \longmapsto \dualita{L}{F(z)}$ is holomorphic in $D$ for every $\ci$-linear continuous 
  $L : X \otimes_\erre \ci \funzione \ci$. \bs
\end{itemize}
\end{Rem}

In the remaining part of this section, \textit{$D$ will denote a nonempty subset of $\ci$ invariant under complex conjugation.}

\begin{Def} \label{def:slice-regular-functions}
A function $F = (F_1,F_2): D \funzione X \otimes_\erre \ci$ is said to be a \emph{stem function} if 
\begin{equation}
  F(\zbar) = \overline{F(z)} \qquad \forall z \in D, \notag
\end{equation}
i.e. $F_1(\zbar) = F_1(z)$ and $F_2(\zbar) = -F_2(z)$ for every $z \in D$.

Let $\Omega_D$ be the circular subset of $Q_\alg$ associated to $D$ and, for every $\j \in \su_\alg$, let 
$\phi_\j:\ci \funzione \ci_\j$ be the isomorphism $\phi_\j(r+si)=r+s\j$ (cf. Definition \ref{circular sets}). We say that 
$f : \Omega_D \funzione X$ is a \emph{($X$-valued) right slice function} if there exists a stem function 
$F = (F_1,F_2): D \funzione X \otimes_\erre \ci$ such that 
\begin{equation}\label{representation of f via F}
  f(\phi_\j(z)) = F_1(z) + F_2(z)\j \qquad \forall z \in D, \quad \forall \j \in \su_\alg. 
\end{equation}
In this case, we write $f = \II_\r(F)$. In the reminder of the paper we will set 
$f_\j := f \circ \phi_\j : D \funzione X_\j$.
\end{Def}

The right slice function $f$ is well-defined and it is induced by a unique stem function. Indeed, if $r \in \erre$, then $F_2(r)=0$ (being $F_2(\overline{z})=-F_2(z)$) and $f(r)=F_1(r)$ independently from the choice of $\j \in \su_\alg$. If 
$q \in Q_\alg \setmeno \erre$, then it admits two representations $q=\phi_\j(z)=\phi_{-\j}(\overline{z})$ with 
$z \in D \setmeno \erre$ and $\j \in \su_\alg$. However, $f(q)$ is uniquely determined by $F$:
\[
f(\phi_\j(z))=F_1(z)+F_2(z)\j=F_1(\overline{z})+F_2(\overline{z})(-\j)=f(\phi_{-\j}(\overline{z})).
\]
The stem function $F$ is in turn uniquely determined by $f$:
\begin{equation} \label{repr. formula 2}
F_1(z)=\frac{1}{2}\left(f(q) + f(q^c)\right),
\qquad
F_2(z)=-\frac{1}{2}\left(f(q) - f(q^c)\right)\j
\end{equation}
if $z \in D$, $\j \in \su_\alg$ and $q=\phi_\j(z)$. The latter equalities imply the following 
\textit{representation formula for right slice functions} $f$:
\begin{equation}\label{repr. formula 1}
  f(r+s\k) = \frac{1}{2}\left(f(q) + f(q^c)\right) - \frac{1}{2}\left(f(q) - f(q^c)\right)\j\k
\end{equation}
if $q=r+s\j \in \Omega_D$, $r,s \in \erre$ and $\j,\k \in \su_\alg$.

We now introduce the notion of slice regularity for vector-valued mappings.

\begin{Def}
Let $D \subseteq \ci$ be open and let $f : \Omega_D \funzione X$ be a right slice function with $f = \II_\r(F)$. We say that 
$f$ is \emph{right slice regular} if $F$ is holomorphic in $D$, i.e. if $F \in C^1(D;X \otimes_\erre \ci)$ and 
\begin{align} \label{eq:s-r-f}
 \frac{\partial F}{\partial r} + i\frac{\partial F}{\partial s} = 0,
\end{align}
where $(r,s)$ denote the real coordinates in $\ci$. If $F=(F_1,F_2)$, then (\ref{eq:s-r-f}) is equivalent to
\begin{equation} \label{eq:c-r-type-system}
 \frac{\partial F_1}{\partial r} = \frac{\partial F_2}{\partial s}, \qquad
 \frac{\partial F_1}{\partial s} = -\frac{\partial F_2}{\partial r}.
\end{equation}
\end{Def}

\begin{Def}
The notions of left slice and left slice regular functions are completely analogous to the right ones. We say that 
$f:\Omega_D \funzione X$ is a \emph{left slice function} if there exists a (unique) stem function 
$F = (F_1,F_2) : D \funzione X \otimes_\erre \ci$ such that $f(\phi_\j(z)) = F_1(z) + \j F_2(z)$ for every $z \in D$ and 
$\j \in \su_\alg$. In this case, we write $f=\II_\ell(F)$. If $D$ is open in $\ci$ and $F$ is holomorphic in $D$, then $f$ is called \textit{left slice regular}. 
\end{Def}

\begin{Example} \label{exa:elementary}
(a)
If $c \in X$, then the constant function $f : \Omega_D \funzione X : q \longmapsto c$ is obviously both left and right regular. In the reminder of the paper we will denote the constant functions by its constant value: $f = c$.

(b)
If $X$ is a Banach two-sided $\alg$-algebra, $c \in X$, and $f : \Omega_D \funzione X$ is right slice regular with 
$f = \II_\r(F)$, then $g : \Omega_D \funzione X : q \longmapsto cf(q)$ is right slice regular, since $g = \II_\r(G)$, where $G(z) = cF(z)$ is a holomorphic stem function. On the other hand, in general $q \longmapsto f(q)c$ is not a right slice function, but it is left slice regular if $f$ is. \bs 
\end{Example}


\begin{Prop}\label{f regular iff f slice holomorphic}
Let $D \subseteq \ci$ be open and let $f : \Omega_D \funzione X$ be a right slice function.
Then the following statements are equivalent.
\begin{itemize}
\item[(i)]
  $f$ is right slice regular.
\item[(ii)]
  $f_\j := f \circ \phi_\j: D \funzione X_\j$ is holomorphic for every $\j \in \su_\alg$.
\item[(iii)]
  There exists $\j \in \su_\alg$ such that $f_\j: D \funzione X_\j$ is holomorphic.
\end{itemize}
\end{Prop}

\begin{proof}
Assume that$f = \II_\r(F)$ with $F = (F_1, F_2)$. Since $f_\j(z) = f(\phi_\j(z)) = F_1(z) + F_2(z)\j$, recalling
\eqref{product defining X_j}, if $f$ satisfies (i), then it holds 
\begin{align}
  \frac{\partial f_\j}{\partial r} + i \frac{\partial f_\j}{\partial s} 
    & = \frac{\partial F_1}{\partial r} + \frac{\partial F_2}{\partial r}\j + 
          \left(\frac{\partial F_1}{\partial s} + \frac{\partial F_2}{\partial s}\j\right)\j \notag \\
    & = \frac{\partial F_1}{\partial r} + \frac{\partial F_2}{\partial r}\j + 
          \frac{\partial F_1}{\partial s}\j - \frac{\partial F_2}{\partial s}\notag \\      
    & = \frac{\partial F_1}{\partial r} - \frac{\partial F_2}{\partial s} + 
          \left(\frac{\partial F_2}{\partial r} + \frac{\partial F_1}{\partial s}\right)\j=0. \notag
\end{align}
This proves (ii). The implication $(\mathrm{ii}) \Longrightarrow (\mathrm{iii})$ is evident. Finally, suppose (iii) holds, i.e. 
$\frac{\partial f_\j}{\partial r}+\frac{\partial f_\j}{\partial s} \j=0$ for some $\j \in \su_\alg$. Thanks to \eqref{repr. formula 2}, we infer that
\begin{align*}
2\frac{\partial F_1}{\partial r}(z) & = \frac{\partial f_\j}{\partial r}(z)+\frac{\partial f_\j}{\partial r}(\overline{z})=
-\frac{\partial f_\j}{\partial s}(z)\j-\frac{\partial f_\j}{\partial s}(\overline{z})\j=2\frac{\partial F_2}{\partial s}(z)
\end{align*}
for every $z \in D$. Similarly, we obtain also the second equality of \eqref{eq:c-r-type-system}, and (i) follows.
\end{proof}

\begin{Rem}
If $\alg = \quat$ and $D \subseteq \ci$ is connected, then Proposition \ref{f regular iff f slice holomorphic} entails that 
a function $f : \Omega_D \funzione \quat$ is slice regular if and only if it is regular in the sense of \cite[Definition 2.2]{GeSt}.
\bs
\end{Rem}

For $X$-valued slice regular functions the following extension lemma holds.

\begin{Lem}\label{slice analytic continuation}
Let $D \subseteq \ci$ be open and connected, and let $f : \Omega_D \funzione X$ be a right slice regular function. If 
$f(q) = 0$ for all $q \in \Omega_D \cap \erre$, then $f = 0$.
\end{Lem}

\begin{proof}
Since $D$ is connected and invariant under complex conjugation, then $D \cap \erre \neq \vuoto$. Assume that 
$f = \II_\r(F)$ with $F = (F_1,F_2)$. Let $r \in D \cap \erre$. Since $F(\zbar) = \overline{F(z)}$, we have
$F_2(r) = 0$. Choose $\j \in \su_\alg$. Then
\[
  0 = f(\phi_\j(r)) = F_1(r) + F_2(r) \j = F_1(r),
\]
thus $F(r) = 0$ for every $r \in D \cap \erre$. Hence $F = 0$ in $D$ and the lemma follows from 
\eqref{representation of f via F}.
\end{proof}

If $X$ is also a Banach two-sided $\alg$-algebra with unit, then one can perform the pointwise product of two right slice functions. However, it is not in general a right slice function. On the contrary, it is immediately seen that the pointwise product of two stem functions is a stem function. Therefore it is possible to define the following notion of slice product which generalizes the convolution product between power series with coefficients in $X$.

\begin{Def}\label{D:slice product}
Let $X$ be a Banach two-sided $\alg$-algebra with unit, and let $f : \Omega_D \funzione X$ and 
$g : \Omega_D \funzione X$ be two right slice functions with $f = \II_\r(F)$, $g = \II_\r(G)$ and $F = F_1+F_2\,\ui$, 
$G = G_1+G_2\,\ui$. The \emph{(right) slice product of $f$ and $g$} is the right slice function $f \cdot g  := \II_\r(FG)$, where $FG$ is the pointwise product of $F$ and $G$:
\begin{align}
  (FG)(z) & := \big(F_1(z) G_1(z) - F_2(z) G_2(z)\big)+\big(F_1(z) G_2(z) + F_2(z) G_1(z)\big)\,\ui \qquad \forall z \in D. \notag
\end{align}
\end{Def}

Since a real bilinear product of two holomorphic vector functions is holomorphic, we immedia\-tely get the following result.

\begin{Prop}
The slice product of two right slice regular functions is right slice regular.
\end{Prop}

Some more words are necessary concerning the notation for the slice product. If we want to stress the role
of the independent variable, the following notation is convenient:
\[
  f(q) \cdot_q g(q) := (f \cdot g)(q).
\]
This is especially useful when the functions $f$ and $g$ depend on several variables:
\[
  f(p,q) \cdot_q g(p,q) := (f(p,\cdot) \cdot g(p,\cdot))(q).
\]
Dealing with slice powers, the following notation will be also used for $n \in \enne$:
\[
  f(p,q)^{\cdot_q n} := \underbrace{f(p,q) \cdot_q f(p,q) \cdot_q \cdots \cdot_q f(p,q)}_{\text{$n$ times}},
\]
i.e. this $n$-th power is the slice product with respect to $q$ of $f(p,q)$ with itself computed $n$ times. 


\section{Examples of vector slice regular functions}\label{S:slice regular examples}

In this section we consider some fundamental examples of vector-valued right slice regular functions, which will be exploited in the remainder of the paper. Throughout the section we assume that
\[
\text{\emph{$\alg$ satisfies \eqref{assumption on A} and \eqref{eq:assumption-2}.}}
\]


\subsection{Right power series}\label{Es funzioni right slice regular}

Let $(c_n)$ be a sequence in a Banach two-sided $\alg$-module $X$. Consider the  series $s=\sum_{n \ge 0} c_n q^n$ with $q \in Q_\alg$. Thanks to Proposition \ref{useful properties}(b) and Lemma \ref{norma moltiplicativa}, we know that 
$\|c_nq^n\|=\|c_n\| |q|^n$. This equality ensures the validity of Abel theorem for $s$. In other words, if $R \in [0,+\infty]$ is defined by $1/R:=\limsup_{n \to +\infty}\sqrt[n]{\|c_n\|}$ and if $R > 0$, then $s$ converges totally on compact subsets of the ball $\Omega_R:=\{q \in Q_\alg \, : \, |q|<R\}$. The sum function 
$\Omega_R \funzione X: q \longmapsto \sum_{n \ge 0} c_n q^n$ of $s$, we denote again by $s$, is right slice regular. Indeed, if $B_R$ is the Euclidean open ball of $\ci$ centered at $0$ of radius $R$ and $S_1,S_2:B_R \funzione X$ are functions defined by $S_1(z) := \sum_{n \ge 0} c_n \Re(z^n)$ and$S _2(z) := \sum_{n \ge 0} c_n \Im(z^n)$, then 
$S = (S_1, S_2)$ is a holomorphic stem function and $s = \II_\r(S)$. We have just seen that 
\emph{convergent power series with left coefficients in $X$ are right slice regular}. In general, convergent power series with right coefficients are not right slice functions, but they are left slice regular functions.


\subsection{Noncommutative exponentials}

In this subsection we introduce some noncommutative generalizations of the complex exponential functions 
$z \longmapsto e^{a z}$, where $a$ is a vector of a complex Banach algebra. In what follows, we also give the definition of the ``slice translation'' of these exponential mappings. First we recall that there exists a positive constant $C$, depending only on $\alg$, such that
\[
|(p+q)^{\cdot_q n}| \le C\big((\re(p) + \re(q))^2 + (|\im(p)| + |\im(q)|)^2\big)^{n/2}
\]
for every $p, q \in Q_\alg$ and for every $n \in \enne$ (cf. \cite[Inequality (3.2)]{GhiPer14}). Therefore the series in the following formula \eqref{exp^x_p} is convergent on the whole $Q_\alg$.

\begin{Def}
Let $X$ be a Banach two-sided $\alg$-algebra, let $x \in X$ and let $p \in Q_\alg$. We define the right slice regular function $\exp_{p}^x : Q_\alg \funzione X$ by setting
\begin{align}\label{exp^x_p}
 \exp_{p}^x(q) := \sum_{n \ge 0} \frac{x^n}{n!} (p+q)^{\cdot_q n}, \qquad 
 q \in Q_\alg.
\end{align}
For $p = 0$ we simply set $\exp^x := \exp_0^x$, i.e.
\begin{equation*}
  \exp^x(q) := \exp_0^x = \sum_{n \ge 0} \frac{x^n}{n!} q^{n}, \qquad 
  q \in Q_\alg.
\end{equation*}
We will also write $e^x := \exp^x(1) = \sum_{n \ge 0} \frac{x^n}{n!}$, i.e. the usual exponential function in $_\erre X$.
\end{Def}

Here are the properties of the ``non-commutative'' exponential.

\begin{Lem}\label{L:the exponential exp_c(q)}
Let $X$ be a Banach two-sided $\alg$-algebra, let $x \in X$ and let $p \in Q_\alg$. Then the following propositions hold.
\begin{itemize}
\item[(i)]
 Let $q \in Q_\alg$ with $p+q \in Q_\alg$. If either $x^2=0$ or $pq = qp$, then
  \begin{equation}\label{exp_q(p) if pq=qp}
    \exp^x_p(q) = \exp^x(p+q).
    \end{equation}
A partial vice versa is true. If $\exp^x_{tp}(tq) = \exp^x(t(p+q))$ for every $t \in \erre$, then either $x^2=0$ or $pq=qp$ or $pq-qp$ is a left zero divisor of $\alg$ (that is, $pq-qp \neq 0$ and $(pq-qp)a=0$ for some $a \in \alg \setmeno \{0\}$).
 \item[$(\mr{i}')$] Assume $xp=px$. Let $q \in Q_\alg$ with $p+q \in Q_\alg$. If either $x^2=0$ or $pq = qp$, then
  \begin{equation}\label{semigroup-law}
    \exp^x(p+q) = \exp^x(p)\exp^x(q).
    \end{equation}
In particular this equality holds if $p,q \in \erre$. A partial vice versa is true. If $\exp^x(t(p+q))=\exp^x(tp)\exp^x(tq)$ for every $t \in \erre$, then either $x^2=0$ or $pq=qp$ or $pq-qp$ is a left zero divisor of $\alg$.
 \item[(ii)]
If $xq=qx$ for some $q \in Q_\alg$, then 
\begin{equation} \label{eq:eee}
\exp^x(q)=e^{xq}.
\end{equation}
A~partial vice versa is true. If $\exp^x(tq)=e^{xtq}$ for every $t \in \erre$, then either $xq=qx$ or $x$ is a left zero divisor of $X$.
 \item[(iii)]
  $\exp_p^x$ is the unique right slice regular function on $Q_\alg$ such that 
  $\exp_p^x(t) = \sum_{n \ge 0} \frac{x^n}{n!} (p+t)^{n}$ for every $t \in \erre$.
\item[(iv)] 
  If $t \in \erre$ then $Q_\alg \funzione X : q \longmapsto \exp_q^x(t)$ is right slice regular. 
\end{itemize}
\end{Lem}

\begin{proof}
Let $q \in Q_\alg$ with $p+q \in Q_\alg$. Since
\[
\exp^x_p(q) - \exp^x(p+q)=\frac{x^2}{2}(pq-qp)+\sum_{n \ge 3}\frac{x^n}{n!}\big((p+q)^{\cdot_q n}-(p+q)^n\big),
\]
if either $x^2=0$ or $pq=qp$ then \eqref{exp_q(p) if pq=qp} holds. Suppose that 
$\mr{E}(t):=\exp_{tp}^x(tq)-\exp^{x}(tp+tq)=0$ for every $t \in \erre$. Observe that
\[
\mr{E}(t) 
  = \frac{t^2x^2}{2} (pq-qp) + t^3h_x(t)
\]
for some continuous (real analytic indeed) function $h_x : \erre \lra X$. Thus we have
\[
0 = \lim_{\erre \ni t \to 0} 2\mr{E}(t)t^{-2} =x^2(pq - qp).
\]
If $x^2 \neq 0$ and $p':=pq-qp \neq 0$, then $p'$ must be a left zero divisor. Otherwise, being $\alg$ 
finite dimensional, there would exist $q' \in \alg$ such that $p'q'=1$ and hence $x^2=(x^2p')q'=0$, which is a contradiction. This completes the proof of point (i).

Let us prove $(\mr{i}')$. Suppose $xp=px$. Thank to this hypothesis, we have that $\exp^x(p+q)-\exp^x(p)\exp^x(q)=x^2y$ for some $y \in X$. Thus, \eqref{semigroup-law} is satisfied if $x^2=0$. If instead $p$ commutes with $q$, then
\[
\exp^x(p+q)-\exp^x(p)\exp^x(q)=\sum_{n \geq 2}\frac{x^n}{n!}\sum_{k=0}^n{n \choose k}p^kq^{n-k}-\sum_{n \geq 2}\sum_{k=0}^n\frac{x^k}{k!}p^k\frac{x^{n-k}}{(n-k)!}q^{n-k}=0.
\]
Suppose that $\mr{F}(t):=\exp^x(tp+tq)-\exp^x(tp)\exp^x(tq)$ for every $t \in \erre$. We have that
\[
\mr{F}(t)=\frac{x^2}{2}(qp-pq)+t^3k_x(t)
\]
for some continuous function $k_x:\erre \lra X$, and hence $0=\lim_{\erre \ni t \to 0}2\mr{F}(t)t^{-2}=x^2(qp-pq)$. We can now conclude as above.

The proof of point (ii) is similar. If $xq=qx$, then \eqref{eq:eee} is evident. Suppose that $\mr{G}(t):=\exp^x(tq)-e^{xtq}=0$ for every $t \in \erre$. Since
\[
\mr{G}(t)=\sum_{n \geq 2}\big(x^nq^n-(xq)^n\big)\frac{t^n}{n!}=\frac{t^2}{2}x(xq-qx)q+t^3\ell_x(t)
\]
for some continuous function $\ell_x:\erre \lra X$, it follows that $0=\lim_{\erre \ni t \to 0}2\mr{G}(t)t^{-2}=x(xq-qx)q$ and hence $0=x(xq-qx)$ if $q \neq 0$. Thus either $xq=qx$ or $x$ is a left zero divisor of $X$.

Point (iii) is a consequence of (i) and Lemma \ref{slice analytic continuation}, while point (iv) follows from (i) and (iii), being 
$\exp^x_q(t)=\exp^x(q+t)=\exp^x_t(q)$.
\end{proof}


\subsection{Operatorial slice composition}
Assume that
\[
\text{\emph{$X$ is a Banach two-sided $\alg$-module.}}
\]
The slice product deserves a particular attention when functions take on values in the set $\Lin^\r(X)$ whose product is the composition of operators. In order to avoid any notational ambiguity we explicitly state we consider the following product
\[
  \Lin^\r(X) \times \Lin^\r(X) \funzione \Lin^\r(X) : (\A, \B) \longmapsto \A\B := \A \circ \B,
\]
i.e. $(\A\B)(x) = \A(\B x)$ for every $x \in X$, which makes $\Lin^\r(X)$ a Banach two-sided $\alg$-algebra with unit. In this special case we will adopt a new symbol for the \emph{slice composition} of operatorial functions, i.e. the slice product of operatorial functions $q \longmapsto \F(q), q \longmapsto \G(q)$ will be denoted with the symbol ``$\odot$'' rather than the dot ``$\cdot$''. For the sake of clarity we formalize this notation in the following definition.

\begin{Def}\label{D: slice composition}
Let $D$ be a nonempty subset of $\ci$ invariant under complex coniugation and let $\Omega_D$ be the circular subset of $Q_\alg$ associated to $D$. Consider two right slice functions $\F : \Omega_D \funzione \Lin^\r(X)$ and 
$\G : \Omega_D \funzione \Lin^\r(X)$ with $\F = \II_\r(\so{F})$, $\G = \II_\r(\so{G})$ and $\so{F} = (\so{F}_1,\so{F}_2)$, 
$\so{G} = (\so{G}_1,\so{G}_2)$. The slice product of $\F$ and $\G$ will be called \emph{(right)} 
\emph{slice composition of $\F$ and $\G$} and will be denoted by $\F \odot \G : \Omega_D \funzione \Lin^\r(X)$. In other terms  
\begin{equation}
  \F \odot \G := \II_\r(\so{F}\so{G}), \notag
\end{equation}
i.e.
\begin{equation}\label{esplicita slice operatorial composition}
  (\F \odot \G)(\phi_\j(z)) =
  (\so{F}_1(z) \circ \so{G}_1(z) - \so{F}_2(z) \circ \so{G}_2(z)) +  
    (\so{F}_2(z) \circ \so{G}_1(z) + \so{F_}1(z) \circ \so{G}_2(z) )\j
\end{equation}
for every $z \in D$ and for every $\j \in \su_{\alg}$.
\end{Def}

\begin{Example} \label{exa:elementary-2}
Let $\A \in \Lin^\r(X)$ and let $\F : \Omega_D \funzione \Lin^\r(X)$ be a right slice function. We know from 
Example \ref{exa:elementary}(b) that the function $q \longmapsto \A \F(q) = \A \circ \F(q)$ is a right slice function. In general, instead, the mapping $q \longmapsto \F(q)\A$ is not right slice. If we consider $\A$ as a constant function, then the slice composition $\F \odot \A$ is a right slice function. Explicitly if $\F = \II_\r(\so{F})$ with 
$\so{F} = (\so{F}_1,\so{F}_2)$ then
\[
  (\F \odot \A)(\phi_\j(z)) = \so{F}_1(z) \A + \so{F}_2(z) \A \, \j \qquad \forall z \in D,\ \forall \j \in \su_\alg. \ \text{\bs}
\]
\end{Example}

In the following remark we show that, given a right slice function $\F : \Omega_D \funzione \Lin^\r(X)$ and $x \in X$, in general we cannot conclude that $\Omega_D \funzione X : q \longmapsto \F(q)x$ is right slice (nor left slice).

\begin{Rem} \label{eq:no-right-slice}
Let $x \in X$, let $\F : \Omega_D \funzione \Lin^\r(X)$ be a right slice function and let $(\so{F}_ 1, \so{F}_2)$ be the stem function inducing $\F$. Define $\F_x : \Omega_D \funzione X$ by $\F_x(q) := \F(q)x$. Therefore we have that 
$\F_x(\phi_\j(z))=(\so{F}_1(z) + \so{F}_2(z)\j)x = \so{F}_1(z)x + \so{F}_2(z)(\j x)$ if $z \in D$ and $\j \in \su_\alg$. Given 
$\j \in \su_\alg$, if $\F_x=\II_r(F)$ for some stem function $F:D \funzione X \otimes_\erre \ci$ with $F=(F_1,F_2)$, then \eqref{repr. formula 2} implies that $F_1(z)=\so{F}_1(z)x$ and $F_2(z)=-\so{F}_2(z)(\j x\j)$. It follows immediately that 
$\F_x$ is right slice if and only if, for every $z \in D$ and for every $\j,\k \in \su_\alg$, $\so{F}_2(z)(\k x)=-\so{F}_2(z)(\j x\j)\k$ or, equivalently, $\so{F}_2(z)(\k x\k-\j x\j)=0$. A concrete example in which the latter equality fails is as follows. Denote by 
$\{1,i,j,k\}$ the standard real vector basis of $\quat$ and define $\alg:=\quat$, $X$ as $\quat$ with its standard structure of Banach two-sided $\alg$-module (cf. Example \ref{exa:b-t-s-a}), $x:=j \in X=\quat$ and the function 
$\F: \quat \funzione \Lin^\r(\quat)$ by setting $\F(q)p:=iqp$ for every $p \in \quat=X$. Observe that, in this case, 
$\so{F}_2(z)p=i\Im(z)p$ and hence, if $z=i \in \ci$, $\j:=j \in \alg=\quat$ and $\k:=k \in \alg=\quat$, then 
$\so{F}_2(z)(\k x\k-\j x\j)=2k \neq 0$. It follows that $\F_x: \quat \funzione \quat$, $\F_x(q)=iqj$, is not right slice. It is immediate to easy that $\F_x$ is not left slice as well. \bs
\end{Rem}


\subsection{Integrals}

Here is a result on the slice regularity of integrals depending on a parameter.

\begin{Prop}\label{slice regularity of integral functions}
Let $X$ be a Banach two-sided $\alg$-module, let $I$ be an interval of $\erre$, let $D$ be a nonempty open subset of $\ci$ invariant under complex conjugation and let $f : I \times \Omega_D \funzione X$ be a map such that 
$f(\cdot,q) \in L^1(I;X)$ for every $q \in \Omega_D$ and $f(t,\cdot)$ is right slice regular with $f(t,\cdot) = \II_\r(F(t,\cdot))$ for every $t \in I$. Suppose there exist $\j \in \su_\alg$ and $g_r, g_s \in L^1(I;\erre)$ such that, if 
$f_\j:I \times D \funzione X$ denotes the map $f_\j(t,z):=f(t,\phi_\j(z))$, $\norma{(\partial f_\j/\partial r)(t,z)}{} \le g_r(t)$ and 
$\norma{(\partial f_\j/\partial s)(t,z)}{} \le g_{s}(t)$ for every $t \in I$ and for every $z=r+is \in D$. Then the function 
$h : \Omega_D \funzione X$ defined by
\begin{equation*}
  h(q) := \int_I f(t,q) \de t, \qquad 
  q \in \Omega_D
\end{equation*}
is right slice regular and $h = \II_\r(H)$, where $H(z) := \int_I F(t,z) \de t$ if $z \in D$.
\end{Prop}

\begin{proof}
If $F(t,\cdot) = (F_1(t,\cdot),F_2(t,\cdot))$ for every $t \in I$, then representation formulas \eqref{repr. formula 2} imply that $F_m(\cdot, z)$, $m = 1, 2$, is integrable for every $z \in D$, therefore the definition of $H$ makes sense. If $H=(H_1,H_2)$, then $H_m(z) = \int_I F_m(t,z) \de t$ for every $z \in D$, $m = 1, 2$, $H$ is a stem function, and 
\[
  H_1(z) + H_2(z)\j = \int_I (F_1(t,z) + F_2(t,z) \j) \de t = \int_I f(t,\phi_\j(z)) \de t = h(\phi_\j(z))
\]
for every $\j \in \su_\alg$ and every $z \in D$, i.e. $h = \II_\r(H)$. Using again formulas \eqref{repr. formula 2} we also infer that there exist $G_{r,m},G_{s,m} \in L^1(I;\erre)$, $m=1,2$, such that $\norma{(\partial F_m/\partial r)(t,z)}{} \leq G_{r,m}(t)$ and $\norma{(\partial F_m/\partial s)(t,z)}{} \leq G_{s,m}(t)$ for every $(t,z) \in I \times D$. In this way we can perform derivatives under the sign of integral, obtaining 
\[
  \frac{\partial H}{\partial r}(z) + i\frac{\partial H}{\partial s}(z) = 
  \int_I \left( \frac{\partial F}{\partial r}(t,z) + i\frac{\partial F}{\partial s}(t,z) \right) \de t = 0,
\]
therefore $H$ is holomorphic and we are done.
\end{proof}


\subsection{Spherical resolvent operator} \label{subsec:5.4}

The notions of spherical spectrum and of spherical resolvent operator was given for the first time in \cite{CoGeSaSt07} for quaternions and in \cite{ColSabStr08} for arbitrary Clifford algebras $\erre_n$.
Here we consider the general case introduced in \cite[Definition 2.26]{GhiRec15}.

Assume that
\[
\text{\emph{$\alg$ satisfies \eqref{assumption on A}, \eqref{S_A nonvuota} and \eqref{eq:assumption-2bis}, and $X$ is a Banach two-sided $\alg$-module.}}
\]

\begin{Def}\label{D:spherical spectral notions}
Let $D(\A)$ be a right $\alg$-submodule of $X$ and let $\A : D(\A) \funzione X$ be a closed right linear operator. Given 
$q \in Q_\alg$, we define the right linear operator $\Delta_q(\A) : D(\A^2) \funzione X$ by setting
\begin{equation*}
  \Delta_q(\A) := \A^2 - 2\re (q) \,\A + |q|^2 \Id.
\end{equation*}
The \emph{spherical resolvent set $\rho_\s(\A)$ of $\A$} and the \emph{spherical spectrum $\sigma_\s(\A)$ of $\A$} are the circular subsets of $Q_\alg$ defined as follows: 
\begin{equation*}
\rho_\s(\A) := \{q \in Q_\alg \, : \, \text{$\Delta_q(\A)$ is bijective, $\Delta_q(\A)^{-1} \in \Lin^\r(X)$}\}
\end{equation*}
and
\begin{equation*}
  \sigma_\s(\A) :=Q_\alg \setmeno \rho_\s(\A).
\end{equation*}
For every $q \in \rho_\s(\A)$, we define the operators $\Q_q(\A) \in \Lin^\r(X)$ and $\C_q(\A) \in \Lin^\r(X)$ by setting
\begin{equation*}
  \Q_q(\A) := \Delta_q(\A)^{-1} 
\end{equation*}
and
\begin{equation}\label{resolvent operator}
\C_q(\A) := \Q_q(\A) q^c - \A \Q_q(\A).
\end{equation}
The operator $\C_q(\A)$ is called \emph{spherical resolvent operator of $\A$ at $q$}.
\end{Def}

Observe that the boundedness of $\C_q(\A)$ follows from the closed graph theorem on $\Lin(_\erre X)$
(cf. \cite[Proposition 2.28]{GhiRec15}). We also mention that a definition that has some similarities with the spherical spectrum was given in \cite{Kap} in the context of real *-algebras.

\begin{Rem}
Let $\alg=\quat$ and let $X = \quat^2$ with standard left and right multiplications by scalars in $\quat$ (cf. Example \ref{exa:b-t-s-a}). Define $\A \in \Lin^\r(\quat^2)$ by setting 
\[
\A
\begin{pmatrix} p \\ q \end{pmatrix}
:=
\begin{pmatrix} 0 &i \\ j & 0 \end{pmatrix}
\begin{pmatrix} p \\ q \end{pmatrix}
=
\begin{pmatrix} iq \\ jp \end{pmatrix}
\qquad \forall
\begin{pmatrix} p \\ q \end{pmatrix} \in \quat^2.
\]
By direct inspection, one easily verifies that $\lambda=\lambda_0+\lambda_1i+\lambda_2j+\lambda_3k \in \quat$ with 
$\lambda_0,\lambda_1,\lambda_2,\lambda_3 \in \erre$ belongs to $\sigma_S(\A)$ if and only if $\lambda_0^2=\frac{1}{2}$ and $\lambda_1^2+\lambda_2^2+\lambda_3^2=\frac{1}{2}$. In other words, we have
\[
\sigma_\s(\A)=\left(-\frac{1}{\sqrt{2}}+\frac{1}{\sqrt{2}} \, \su_\quat\right) \cup \left(\frac{1}{\sqrt{2}}+\frac{1}{\sqrt{2}} \, 
\su_\quat\right).
\]
It is also immediate to see that the operator $\lambda \Id-\A \in \Lin^\r(\quat^2)$ is not invertible if and only if 
$\lambda \in \{\mu,\mu^c\}$, where $\mu:= \frac{1}{\sqrt{2}}(i+j)$. Observe that 
$\{\mu,\mu^c\} \cap \sigma_\s(\A)=\vuoto$. This shows that in general there is no relation between the notion of spherical spectrum and the noncommutative version~of the classical concept of spectrum. Indeed, $\Delta_\mu(\A)$ is invertible, but $\mu\Id-\A$ and $\mu^c\Id - \A$ are~not. Moreover, if $\lambda=\frac{1}{\sqrt{2}}(1+i)$, then $\lambda\Id-\A$ and $\lambda^c\Id - \A$ are invertible, but $\Delta_\lambda(\A)$ is not.

It is worth recalling that, in the quaternionic matricial case, the spherical spectrum coincides with the set of right eigenvalues (cf. \cite[Proposition 4.5]{GhiMorPer13}): $\lambda \in \sigma_\s(\A)$ if and only if $\A x=x\lambda$ for some 
$x \in \quat^2 \setmeno \{(0,0)\}$. The spherical spectrum is equal also to the set of left eigenvalues, provided $\quat^2$ is endowed with a suitable left scalar multiplication (cf. \cite[Example 7.3]{GhiMorPer-bis}). \bs
\end{Rem}

In our next result, given $\j \in \su_\alg$, we describe the deep connection existing between the notions of spherical resolvent set and of spherical resolvent operator of an operator $\A$ on $X$ and the classical complex ones of resolvent set and of resolvent operator of the operator $\A_\j$ on $X_\j$.


\begin{Thm}\label{legame risolventi}
Let $\A : D(\A) \funzione X$ be a closed right linear operator, $\j \in \su_\alg$, and let $\rho(\A_\j)$ denote the resolvent set of the operator $\A_\j : D(\A_\j) \funzione X_\j$ (cf. Definition \ref{D: complex space X_j}). Then the following equivalent assertions hold.
\begin{itemize}
 \item[(i)]
Given $\lambda \in \ci$, we have that $\phi_\j(\lambda) \in \rho_\s(\A)$ if and only if both $\lambda$ and 
$\overline{\lambda}$ belong to $\rho(\A_\j)$.
 \item[(ii)]
$\rho_\s(\A)$ is equal to the circular subset of $Q_\alg$ associated to 
$\rho(\A_\j) \cap \overline{\rho(\A_\j)}$, where $\overline{\rho(\A_\j)}$ denotes the set 
$\{\lambda \in \ci : \overline{\lambda} \in \rho(\A_\j)\}$.
\end{itemize}
Furthermore, if $\lambda \in \rho(\A_\j) \cap \overline{\rho(\A_\j)}$, 
\begin{equation}\label{Q = R R}
  \Q_{\phi_\j(\lambda)}(\A)  = \R_{\overline{\lambda}}(\A_\j) \R_{\lambda}(\A_\j).
\end{equation}
\end{Thm}

\begin{proof}
The equivalence between (i) and (ii) is evident. Let us prove (i). If $\lambda = r + si \in \ci$ with $r, s \in \erre$ and 
$q:= \phi_\j(\lambda) \in Q_\alg$, then for every $x \in D(\A^2)$ we have 
\begin{align}
  (\lambda \Id_{X_\j} - \A_\j)(\overline{\lambda}\Id_{X_\j} - \A_\j)x
    & = ((r+si)\Id_{X_\j} - \A_\j)((r - si)\Id_{X_\j} - \A_\j)x \notag \\
    & = (\A_\j^2 - 2r\A_\j + r^2\Id_{X_\j} + s^2\Id_{X_\j})x \notag \\
    & = \A^2x - \A x(2r)+ x(r^2 + s^2) = \Delta_q(\A)x. \label{Delta = product}
\end{align}
Suppose that $\{\lambda,\overline{\lambda}\} \subseteq \rho(\A_\j)$. If $y \in D(\A)$, then there is $x \in D(\A)$ such that 
$\A_\j x = \overline{\lambda} x - y$, therefore 
$\A x \in D(\A)$ since $D(\A)$ is a complex vector space. This proves that $(\overline{\lambda}\Id - \A_\j)(D(\A^2)) = D(\A)$, which together with \eqref{Delta = product} implies that $\Delta_q(\A)$ is onto $X$. Moreover from \eqref{Delta = product} we also infer that $\Delta_q(\A)$ is injective and \eqref{Q = R R} holds. This also implies that $\Q_q(\A)$ is continuous. Moreover $\Q_q(\A)$ is right linear by virtue of \cite[Lemma 2.16]{GhiRec15}, thus $q \in \rho_\s(\A)$. Suppose now that 
$q \in \rho_\s(\A)$. From \eqref{Delta = product} it follows that 
$(\lambda \Id_{X_\j} - \A_\j)(\overline{\lambda}\Id_{X_\j} - \A_\j) \Q_q(\A)x = x$ for every $x \in X$, thus 
$\lambda \Id_{X_\j} - \A_\j$ has a right inverse, which is provided by the operator
\[
  (\overline{\lambda}\Id_{X_\j} - \A_\j) (\Q_q(\A))_\j
\]
Let us prove that it is also a left inverse. Consider a point $y$ in $D(\A_\j)=D(\A)$. Since $\Delta_q(\A)(D(\A^3))=D(\A)$
 and $\Delta_q(\A)\A=\A\Delta_q(\A)$, we infer that 
$\A\Q_q(\A)y=\Q_q(\A)\A y$, indeed: 
$\A\Q_q(\A)y=\Q_q(\A)\Delta_q(\A)\A\Q_q(\A)y=\Q_q(\A)\A\Delta_q(\A)\Q_q(\A) y=\Q_q(\A)\A y$. It follows that 
\begin{align*}
  (\overline{\lambda}\Id_{X_\j} - \A_\j) (\Q_q(\A))_\j(\lambda\Id_{X_\j}-\A_\j)y 
    & = (\overline{\lambda}\Id_{X_\j} - \A_\j)\big((\Q_q(\A))_\j(\lambda y) - \Q_q(\A) \A y\big) \notag \\
    & = (\overline{\lambda}\Id_{X_\j} - \A_\j)(\lambda(\Q_q(\A)(y)) - \A \Q_q(\A) y) \notag \\
    & = |\lambda|^2\Q_q(\A)y - \overline{\lambda} (\A \Q_q(\A)y)-\lambda \A_\j (\Q_q(\A)y) + \A \A \Q_q(\A) y \notag \\
    & = (|\lambda|^2  - 2 \Re(\lambda) \A - \A^2)\Q_q(\A) y \notag \\
    & = (|q|^2  - 2 \re(q) \A - \A^2)\Q_q(\A) y \notag \\
&= \Delta_q(\A)\Q_q(\A)y=y. \notag
\end{align*} 
This proves that $\lambda \in \rho(\A_\j)$ and $(\overline{\lambda}\Id_{X_\j} - \A_\j) (\Q_q(\A))_\j=\R_{\lambda}(\A_\j)$. Interchanging $\lambda$ and $\overline{\lambda}$ we obtain also that $\overline{\lambda} \in \rho(\A_\j)$ and we are done.
\end{proof}

\begin{Rem}
Let $\A : D(\A) \funzione X$ be a closed right linear operator and let $\j \in \su_\alg$. Choose 
$\lambda \in \rho(\A_\j) \cap \overline{\rho(\A_\j)}$ and define $q:=\phi_\j(\lambda) \in \rho_\s(\A)$. Since 
$(\overline{\lambda}\Id_{X_\j}-\A_\j)\R_{\overline{\lambda}}(\A_\j)=\Id_{X_\j}$, we have that 
$\A_\j\R_{\overline{\lambda}}(\A_\j)-\Id_{X_\j}=\overline{\lambda}\R_{\overline{\lambda}}(\A_\j)$. In this way, thanks to \eqref{Q = R R}, given any $x \in X$, it holds:
\begin{align*}
\C_q(\A)x-\R_\lambda(\A_\j)x & =\R_{\overline{\lambda}}(\A_\j) \R_{\lambda}(\A_\j)(q^cx)-(\A_\j\R_{\overline{\lambda}}(\A_\j)-\Id_{X_\j})\R_\lambda(\A_\j)x \\
& =\R_{\overline{\lambda}}(\A_\j) \R_\lambda(\A_\j)(q^cx)-\overline{\lambda}\R_{\overline{\lambda}}(\A_\j)\R_\lambda(\A_\j)x \\
& =\R_{\overline{\lambda}}(\A_\j) \R_\lambda(\A_\j)(q^cx-xq^c).
\end{align*}
It follows that
\begin{align*}
 & \forall x \in X \;\; : \;\; \C_q(\A)x=\R_\lambda(\A_\j)x \quad \Longleftrightarrow \quad q^cx=xq^c,\\
 & \C_q(\A)=\R_\lambda(\A_\j) \quad \Longleftrightarrow \quad q^cx=xq^c \quad \forall x \in X.
\end{align*}
In particular, we have that $\C_r(\A)=\R_r(\A_\j)$ for every $r \in \erre$. \bs
\end{Rem}

\begin{Prop}\label{q -> C_q slice regular}
Let $\A : D(\A) \funzione \X$ be a closed right linear operator such that 
$\rho_\s(\A) \cap \erre \neq \vuoto$. Then the mapping $\rho_\s(\A) \funzione \Lin^\r(X) : q \longmapsto \C_q(\A)$ is right slice regular. 
\end{Prop}

\begin{proof}
Following the proof of Lemma 2.36 of \cite{GhiRec15}, one obtains that, if $\lambda \in \rho_\s(\A) \cap \erre$, 
$\B:=-\C_\lambda(\A) \in \Lin^\r(X)$ and $\Phi:Q_\alg \setmeno \{\lambda\} \funzione Q_\alg \setmeno \{\lambda\}$ is the inversion map $\Phi(q):=(q-\lambda)^{-1}$, then $\Phi(\rho_\s(\A) \setmeno \{\lambda\})=\rho_\s(\B) \setmeno \{\lambda\}$ and $\C_q(\A)=-\B\C_{\Phi(q)}(\B)\Phi(q)$ for every $q \in \rho_\s(\A) \setminus \{\lambda\}$. Since \cite[Lemma 2.31]{GhiRec15} ensures that $\rho_\s(\B)$ is nonempty and open in $Q_\alg$, we infer that  $\rho_\s(\A)$ is a nonempty open circular subset of $Q_\alg$. Let $D$ be the nonempty open subset of $\ci$ invariant under complex conjugation such that 
$\Omega_D=\rho_\s(\A)$. Fix $\j \in \su_\alg$ and define the stem function
$\so{F} = (\so{F_1},\so{F_2}) : D \funzione \Lin^\r(X) \otimes_\erre \ci$ as follows:
\begin{align}
  & \so{F_1}(z) := \Q_{\phi_\j(z)}(\A) \Re(z) - \A \Q_{\phi_\j(z)}(\A), \label{stem risolv sf 1} \\
  & \so{F_2}(z) := -\Q_{\phi_\j(z)}(\A) \Im(z).
  \label{stem risolv sf 2}
\end{align}
For every $z \in D$, we have that $\so{F}(\zbar) = \overline{\so{F}(z)}$ and
\begin{align}
  \so{F_1}(z) + \so{F_2}(z) \j
  &= \Q_{\phi_\j(z)}(\A)\Re(z)-\A\Q_{\phi_\j(z)}(\A) - \Q_{\phi_\j(z)}(\A)\Im(z)\j \notag \\
   & = \Q_{\phi_\j(z)}(\A){\phi_\j(z)}^c - \A\Q_{\phi_\j(z)}(\A) = \C_{\phi_\j(z)}(\A), \label{useful}
\end{align}
therefore $q \longmapsto \C_q(\A)$ is right slice. From \cite[Lemma 2.32]{GhiRec15} it follows that the map 
$D \funzione \left(\Lin^\r(X)\right)_\j: z \longmapsto \C_{\phi_\j(z)}(\B)$ is holomorphic. Since 
$\C_q(\A)=-\B\C_{\Phi(q)}(\B)\Phi(q)$, the map $z \longmapsto \C_{\phi_\j(z)}(\A)$ is holomorphic as well. Now we can conclude by invoking 
Proposition \ref{f regular iff f slice holomorphic}.
\end{proof}

If $\omega \in \erre$ a straightforward computation shows that $\Delta_{q}(\A - \omega \Id) = \Delta_{q+\omega}(\A)$ for every $q \in Q_\alg$, therefore we can relate the spherical resolvent operators of $\A$ and of $\A - \omega\Id$ as in the classical case (this is not true if $\omega$ is not real).

\begin{Lem}\label{A <-> A - w}
If $\omega \in \erre$ then $\Delta_{q}(\A - \omega \Id) = \Delta_{q+\omega}(\A)$ for every $q \in Q_\alg$ and 
$\rho_\s(\A - \omega\Id) = \rho_\s(\A) - \omega$. Moreover $\Q_{q}(\A - \omega \Id) = \Q_{q+\omega}(\A)$ and
$\C_{q}(\A - \omega \Id) = \C_{q+\omega}(\A)$ for every $q \in \rho_\s(\A - \omega \Id)$.
\end{Lem}

\section{Right linear operator semigroups}\label{semigroups}


Throughout this section, we will assume that 
\begin{equation}\label{hp. for sect. semigroups}
  \text{\emph{$\alg$ satisfies \eqref{assumption on A}, \eqref{S_A nonvuota} and \eqref{eq:assumption-2bis}, and $X$ is a Banach two-sided $\alg$-module.}} \notag
\end{equation}


\subsection{Strongly continuous semigroups}

We first recall the natural definition of right linear operator semigroup (cf. \cite{ColSab11} for the quaternionic case and  \cite{GhiRec15} for the general case).

\begin{Def}\label{def semigroup}
A mapping $\T : \clsxint{0,\infty} \funzione \Lin^\r(X)$ is called \emph{(right linear operator)} \emph{semigroup}~if
\begin{align*}
  & \T(t+s) = \T(t)\T(s) \qquad \forall t, s > 0, \\
  & \T(0) = \Id.
\end{align*}
A semigroup $\T$ is called \emph{uniformly continuous} if $\T \in C(\clsxint{0,\infty};\Lin^\r(X))$. 
A semigroup $\T$ is called \emph{strongly continuous} if  $\T(\cdot)x \in C(\clsxint{0,\infty};X)$ for every $x \in X$. The \emph{generator of $\T$} is the right linear operator $\A : D(\A) \funzione X$ defined by
\begin{align*}
  & D(\A) := \big\{x \in X\ :\ \exists \lim_{h \to 0} (1/h)(\T(h)x - x) = (\de/\de t) \T(t)x\big|_{t=0}\big\}, \\
  & \A x := \lim_{h \to 0}\frac{1}{h}(\T(h)x - x), \qquad 
  x \in D(\A).
\end{align*}
\end{Def}

\begin{Rem}\label{Rem real semigroups}
By Lemma \ref{L:L^r chiuso in L}, we could also say that $\T$ is a uniformly continuous (resp. strongly continuous) semigroup in $X$ if and only if $\T$ is $\Lin^\r(X)$-valued and $\T$ is a uniformly continuous (resp. strongly continuous) semigroup in $_\erre X$. \bs
\end{Rem}

Here is the generation theorem relating generators and semigroups (cf. \cite[Section 4]{ColSab11} for the quaternionic case and \cite[Theorem 4.5]{GhiRec15} for the general case).

\begin{Thm}\label{T:A-generation thm}
The following assertions hold.
\begin{itemize}
\item[$(\mr{a})$]
Let $\A : D(\A) \funzione X$ be a closed right linear operator with $D(\A)$ dense in $X$. Suppose that there are constants $M \in \clsxint{1,\infty}$ and $\omega \in \erre$ such that $\opint{\omega,\infty} \subseteq \rho_\s(\A)$ and
  \begin{equation}\label{cond on R^n-A}
    \norma{\C_\lambda(\A)^n}{} \le \frac{M}{(\lambda - \omega)^{n}} \qquad 
    \forall n \in \enne, \quad \forall \lambda > \omega.
  \end{equation}
  Then $\A$ is the generator of the strongly continuous semigroup $\T : \clsxint{0,\infty} \funzione \Lin^\r(X)$ defined by 
  \begin{equation*}
     \T(t)x = \lim_{n \to \infty} e^{t\A_{n}}x, \quad x \in X, \qquad \text{where } \A_{n} := n\A\C_n(\A) \in \Lin^\r(X).
  \end{equation*}
  Moreover, $\norma{\T(t)}{} \le Me^{\omega t}$ for all $t \ge 0$.
\item[$(\mr{b})$]
 Let $\T : \clsxint{0,\infty} \funzione \Lin^\r(X)$ be a strongly continuous semigroup such that there are constants 
 $M \in \clsxint{1,\infty}$ and $\omega \in \erre$ with the following property: $\norma{\T(t)}{} \le Me^{t\omega}$ for all 
 $t \ge 0$. Then the generator $\A$ of $\T$ is closed, $D(\A)$ is dense in $X$, $\opint{\omega,\infty} \subseteq \rho_\s(\A)$ and $\norma{\C_\lambda(\A)^n}{} \le \frac{M}{(\lambda - \omega)^{n}}$ for all $n \in \enne$ and for all $\lambda > \omega$. 
\end{itemize}
In both cases $(\mr{a})$ and $(\mr{b})$, we have that
\begin{equation*}
  \C_\lambda(\A)x = \int_{0}^{\infty}e^{-t\lambda} \T(t)x \de t \qquad \forall \lambda > \omega, \quad \forall x \in X.
\end{equation*}
\end{Thm}

\subsection{Noncommutative Laplace transform}\label{SS:laplace transform}
In the noncommutative setting a natural notion of argument of a number is provided by the following definition
(cf. \cite[Definition 5.1]{GhiRec15}) that we will use also in the complex case.

\begin{Def}
Define the \emph{argument function} 
$\arg : Q_\alg \setmeno \{0\} \lra \clint{0,\pi}$ on $Q_\alg$ as follows. If $q \in Q_\alg \setmeno \erre$, then there exist, and are unique, $\j \in \su_\alg$, $\rho \in \opint{0,\infty}$ and $\theta \in \opint{0,\pi}$ such that 
$q = \rho e^{\theta\j} \in \mathbb{C}_\j$. Thus, we define $\arg(q):=\theta$. Moreover we set: $\arg(q) := 0$ if 
$q \in \opint{0,\infty}$ and $\arg(q):=\pi$ if $q \in \opint{-\infty,0}$. 
\end{Def}

We need to introduce the following classes of open subsets of $\ci$ and of $Q_\alg$.

\begin{Def}
If $\eta \in \cldxint{0,\pi}$ we define
\begin{equation*}
  D_\eta := \{z \in \ci \setmeno \{0\}\ :\ \arg(z) < \eta\}
\end{equation*}
and the associated circular set
\begin{equation*}
  \Omega_\eta := \Omega_{D_\eta} = \big\{ q \in Q_\alg \setmeno \{0\} \, : \, \arg(q) < \eta \big\}.
\end{equation*}
\end{Def}

Now we prove that the spherical resolvent operator can be written as a suitable Laplace transform. This result is stated in \cite[Theorem 4.2]{ColSab11} in the quaternionic setting and for $n=1$. Here we provide a different proof which also allows to get an integral representation for all the integer slice powers of the resolvent operator, confirming the central role of the slice composition defined in Definition \ref{D: slice composition}.
It is worth noting that Proposition \ref{slice regularity of integral functions} does not apply since the mapping
$q \longmapsto (\T(t)e^{-tq})x$ is not right slice regular even if $q \longmapsto \T(t)e^{-tq}$ is (cf. Remark 
\ref{eq:no-right-slice}).

\begin{Thm}\label{re q > 0 in rho(A)}
Let $\T : \clsxint{0,\infty} \funzione \Lin^\r(X)$ be a strongly continuous semigroup. Suppose there exist 
$M \in \clsxint{1,\infty}$ and $\omega \in \erre$ such that $\norma{\T(t)}{} \le Me^{\omega t}$ for all $t \ge 0$. If $\A$ is the generator of $\T$ then $\omega + \Omega_{\pi/2} \subseteq \rho_\s(\A)$ and, for every $n \in \enne$, we have
\begin{equation}\label{Res = Lap}
  (\C_q(\A)^{\odot_q^n})x = 
  \frac{1}{(n-1)!}\int_0^\infty (\T(t)t^{n-1}e^{-tq})x \de t \qquad \forall q \in \omega + \Omega_{\pi/2}, \quad \forall x \in X.
\end{equation}
In particular
\begin{equation}\label{stima risolvente con re(q)}
  \norma{\C_q(\A)^{\odot_q^n}}{} \le \frac{M}{(\re(q) - \omega)^n} \qquad \forall q \in \omega + \Omega_{\pi/2},
  \quad \forall n \in \enne.
\end{equation}
\end{Thm}

\begin{proof}
Fix $\j \in \su_\alg$ and let $\S_\j : \clsxint{0,\infty} \funzione \Lin(X_\j)$ be defined by $\S_\j(t) := \T(t)$. It follows that $\S_\j$ is a strongly continuous semigroup satisfying the estimate $\norma{\S_\j(t)}{} \le Me^{\omega t}$ for every $t \ge 0$ and its generator is the operator $\A_\j : D(\A_\j) \funzione X_\j$ defined by $D(\A_\j) := D(\A)$ and $\A_\j x := \A x$ for every $x \in D(\A_\j)$. Therefore from the classical theory we have that 
$\omega + D_{\pi/2} \subseteq \rho(\A_\j)$, the resolvent set of $\A_\j$, hence $\omega + \Omega_{\pi/2} \subseteq \rho_\s(\A)$ by virtue of Theorem \ref{legame risolventi}.

If $n \in \enne$, $q \in \omega + \Omega_{\pi/2}$, and $x \in X$ are fixed, then a standard $2\eps$-argument shows that $t \longmapsto (\T(t)t^{n-1}e^{-tq})x = \T(t)(t^{n-1}e^{-tq}x)$ is continuous.
Moreover we have that 
$\norma{\T(t)(t^{n-1}e^{-tq}x)}{} \le Mt^{n-1}e^{t(\omega - \re(q))}\norma{x}{}$ for every 
$t \ge 0$, therefore we can define the following $X$-valued Lebesgue integral 
\begin{equation}\label{Lap(q)x}
 \Lap_n(q)x :=
 \int_0^\infty (\T(t)t^{n-1}e^{-tq})x \de t, \qquad 
 q \in \omega + \Omega_{\pi/2}, \quad 
 x \in X.
\end{equation} 
Now we show that \eqref{Lap(q)x} defines a right slice regular function 
$\Lap_n: \omega + \Omega_{\pi/2} \funzione \Lin^\r(X)$. From
the right linearity of $\T(t)$ and from the definition of $X$-valued Lebesgue integral it follows that 
$\Lap_n(q)$ is right linear, moreover
$\norma{\Lap_n(q)x}{} \le \int_0^\infty \norma{\T(t)t^{n-1}e^{-tq}x}{} \de t \le 
\norma{x}{}M \int_0^\infty t^{n-1}e^{t(\omega - \re(q))}\de t$, thus actually
$\Lap_n(q) \in \Lin(_\erre X)$ for every $q \in \omega+\Omega_{\pi/2}$. By a direct computation, it is immediate to verify that $\int_0^\infty t^{n-1}e^{t(\omega - \re(q))}\de t \le (n-1)!(\re(q)-\omega)^{-n}$. In particular, we have that
\begin{equation} \label{eq:stima-Ln}
\norma{\Lap_n(q)}{} \leq \frac{M(n-1)!}{(\re(q)-\omega)^n}.
\end{equation}
For every fixed $t \in \opint{0,\infty}$, let $F_n^{t} = (F_{n,1}^{t},F_{n,2}^{t}):\ci \funzione \alg \otimes_\erre \ci$ be the stem function such that $t^{n-1}\exp^{-t} = \II_\r(F_n^{t})$: $F_{n,1}(z):=t^{n-1}e^{-t\Re(z)}\cos(t\Im(z))$ and 
$F_{n,2}(z):=-t^{n-1}e^{-t\Re(z)}\sin(t\Im(z))$. Thanks to \eqref{repr. formula 2} it makes sense to define 
$\mathcal{L}_{n,k} : \omega+D_{\pi/2} \funzione \Lin^\r(X)$, $k = 1, 2$, by
\[
  \mathcal{L}_{n,k}(z)x := \int_0^\infty (\T(t)F_{n,k}^{t}(z))x \de t, \qquad 
  z \in \omega+D_{\pi/2}, \quad 
  x \in X.
\]
Then $\mathcal{L}_n := (\mathcal{L}_{n,1}, \mathcal{L}_{n,2})$ is a stem function and for every $\j \in \su_\alg$, 
$z \in \omega+D_{\pi/2}$ and $x \in X$, we have that
\begin{align}
  (\mathcal{L}_{n,1}(z) + \mathcal{L}_{n,2}(z)\j)x
    & = \int_0^\infty \big(\T(t)F_{n,1}^{t}(z)x + \T(t)F_{n,2}^{t}(z)\j x\big) \de t \notag \\
    & = \int_0^\infty \T(t)(F_{n,1}^{t}(z) + F_{n,2}^{t}(z)\j)x \de t \notag \\
    & = \int_0^\infty \T(t)t^{n-1}e^{-t\phi_\j(z)} x \de t = \Lap_n(\phi_\j(z))x, \notag
\end{align}
therefore $\Lap_n = \II_\r(\mathcal{L}_n)$ is a right slice function.

Consider $\j \in \su_\alg$ and the map $(\Lap_n)_\j:\omega+D_{\pi/2} \funzione (\Lin^\r(X))_\j$ defined by setting 
$(\Lap_n)_\j:=\Lap_n \circ \phi_\j$. Let us show that $(\Lap_n)_\j$ is of class $C^1$. Denote by $(r,s)$ the real coordinates in $\ci$, and by $\partial_r$ and $\partial_s$ the partial derivatives $\partial/\partial r$ and $\partial/\partial s$, respectively. Observe that, since $\ci \funzione \ci : z \longmapsto e^{-tz}$ is  holomorphic, we have that 
$\partial_r e^{-t\phi_\j(z)}+\partial_s e^{-t\phi_\j(z)}\j=0$ for every $z \in \ci$. Define the mappings 
$\mathsf{D}_{n,\j,r}, \mathsf{D}_{n,\j,s} : \omega+D_{\pi/2} \funzione (\Lin^\r(X))_\j$ by setting
\[
  \mathsf{D}_{n,\j,r}(z)x := \int_0^\infty \T(t)t^{n-1}(\partial_r e^{-t\phi_\j(z)})x \de t,
  \quad 
  \mathsf{D}_{n,\j,s}(z)x := \int_0^\infty \T(t)t^{n-1}(\partial_s e^{-t\phi_\j(z)})x \de t,
\]
for all $x \in X$, $z \in \omega+D_{\pi/2}$. A $2\eps$-argument shows again that $\mathsf{D}_{n,\j,r}$ and 
$\mathsf{D}_{n,\j,s}$ are continuous. Moreover, for every $z \in \omega+D_{\pi/2}$ and for every $h \in \erre \setmeno \{0\}$ such that $z + h \in \omega+D_{\pi/2}$, we find
\begin{align}\label{stima per D_rH}
   & \sp \left\|\frac{(\Lap_n)_\j(z+h) - (\Lap_n)_\j(z)}{h} - \mathsf{D}_{n,\j,r}(z)\right\| \notag \\
    & = \sup_{\norma{x}{}\le 1} 
             \left\|
               \int_0^\infty  \T(t)t^{n-1}\left(\frac{e^{-t\phi_\j(z+h)} - e^{-t\phi_\j(z)}}{h} - \partial_r e^{-t\phi_\j(z)}\right)x \de t\right\| \notag \\
    & \le \sup_{\norma{x}{}\le 1}\int_0^\infty 
            \left\|  \T(t)t^{n-1}
            \left(\frac{e^{-t\phi_\j(z+h)} - e^{-t\phi_\j(z)}}{h} - \partial_r e^{-t\phi_\j(z)}\right)x \right\| \de t 
            \notag \\
    & \le \int_0^\infty 
            M t^{n-1}e^{\omega t}\left|\frac{e^{-t\phi_\j(z+h)} - e^{-t\phi_\j(z)}}{h} - \partial_r e^{-t\phi_\j(z)}\right|  \de t,
\end{align}
where the last integral is finite because
\begin{align}
  & \sp M t^{n-1}e^{\omega t} \left|\frac{e^{-t(r+h+s\j)} - e^{-t(r+s\j)}}{h}- \partial_r e^{-t(r+s\j)}\right|  \notag \\
  &   = M t^{n-1}e^{\omega t}|e^{-t(r+s\j)}| \left|\frac{e^{-th} - 1}{h} + t\right| 
    \le M t^{n-1}e^{t\omega} e^{-tr} 2t = 2 M t^n e^{t(r - \omega)}. \notag 
\end{align}
Moreover the last integrand in \eqref{stima per D_rH} converges to zero as $h \to 0$, therefore we can apply the dominated convergence theorem, obtaining that $\partial_r(\Lap_n)_\j= \mathsf{D}_{n,\j,r}$. The proof that
$\partial_s(\Lap_n)_\j= \mathsf{D}_{n,\j,s}$ is entirely analogous. It follows that 
$(\Lap_n)_\j \in C^1(\omega+D_{\pi/2};(\Lin^\r(X))_\j)$. Moreover, for every $x \in X$, we have
\begin{align}
  \big(\partial_r(\Lap_n)_\j(z) + i\partial_s(\Lap_n)_\j(z)\big) x
  & = \partial_r(\Lap_n)_\j(z)x + \partial_s(\Lap_n)_\j(z)(\j x) 
   = \mathsf{D}_{n,\j,r}x+\mathsf{D}_{n,\j,s}(\j x) \notag \\
   & = \int_0^\infty \big(\T(t)t^{n-1}(\partial_r e^{-t\phi_\j(z)} + \partial_s e^{-t\phi_\j(z)}\j)x\big) \de t  = 0.\notag
\end{align}
From Proposition \ref{f regular iff f slice holomorphic} we infer that $\Lap_n$ is right slice regular. Thanks to Proposition \ref{q -> C_q slice regular} the function 
$q \longmapsto \C_q(\A)^{\odot_q^n}$ is right slice regular as well, and from \cite[Theorem 4.5]{GhiRec15} we have that 
$\C_r(\A) = \R_r(\A) = \Lap_1(r)$ for every $r > \omega$. Furthermore, from \eqref{stem risolv sf 1}, 
\eqref{stem risolv sf 2}, \eqref{useful} and from \eqref{esplicita slice operatorial composition}, it follows that the value of $\C_q(\A)^{\odot_q^n}$ at $q=r$, which we denote by $\C_r(\A)^{\odot_r^n}$, coincides with 
$\C_r(\A)^n$ for every $r > \omega$. From the classical semigroup theory applied to 
$\T : \clsxint{0,\infty} \funzione \Lin(_\erre X)$, we know that $(n-1)!\R_r(\A)^n = \Lap_n(r)$ (cf. 
Remark \ref{Rem real semigroups} and \cite[Corollary 1.11, p. 56]{EngNag00}). Therefore 
$(n-1)!\C_r(\A)^{\odot_r^n} = \Lap_n(r)$ for every $r > \omega$, thus Lemma \ref{slice analytic continuation} implies \eqref{Res = Lap}, i.e. $(n-1)!\C_q(\A)^{\odot_q^n} = \Lap_n(q)$. Estimate \eqref{stima risolvente con re(q)} is now an immediate consequence of \eqref{eq:stima-Ln}.
\end{proof}


\subsection{Uniformly continuous semigroups}

Strongly continuous semigroups are somehow the less regular class of semigroups. At the other extreme there are the 
uniformly continuous semigroups. We have the following result.

\begin{Thm}
Let 
$\T : \clsxint{0,\infty} \funzione \Lin^\r(X)$ be a strongly continuous semigroup and let $\A$ be its generator. Then 
$\T$ is a uniformly continuous semigroup if and only if $\A \in \Lin^\r(X)$. In this case $\T(t) = e^{t\A}$
for every $t \ge 0$ and the mapping $\exp^\A : Q_\alg \funzione \Lin^\r(X)$ defined by
\[
  \exp^\A(q) := \sum_{n \ge 0}\frac{\A^n}{n!}q^n, \qquad 
  q \in Q_\alg,
\]
is the unique right slice regular extension of $\T$.
\end{Thm}

\begin{proof}
The fact that $\A$ is bounded if and only if $\T(t)  = e^{tA}$ for every $t \ge 0$ is proved in \cite[Thereom 4.3]{GhiRec15}. 
The last statement follows immediately from Lemma \ref{L:the exponential exp_c(q)}$(\mr{i}')$ and Lemma 
\ref{slice analytic continuation}.
\end{proof}


\subsection{Slice regular semigroups}

While uniformly continuous semigroups admits a power series representation by means of the exponential function, in the strongly continuous case this representation is not possible, since the generator is not bounded. As in the classical complex case it is possible to develop a quaternionic functional calculus (cf. \cite{ColSabStr08}) that allows to represent $e^{t\A}$ via a Cauchy integral formula if $\A$ is bounded. However the counterpart of this functional calculus for unbounded operators (cf. \cite{CoGeSaSt10}) does not apply to the exponential function $e^{tA}$, i.e. to semigroups. Nevertheless in
\cite{GhiRec15} we show that a Cauchy integral formula representation is possible if the generator of the semigroup is \emph{spherical sectorial}, a natural quaternionic generalization of complex sectorial operators. Now we recall the definition of spherical sectorial operator and in the next section we are going to prove that the semigroups generated by spherical sectorial operators are exactly those who can be extended to a right slice regular operatorial functions on a spherical sector of $\alg$. Moreover this extension satisfies a suitable ``noncommutative semigroup law'', originating what we call the class of \emph{right slice regular semigroup}. This result casts a bridge between
the theory of semigroups on Banach two-sided $\alg$-modules and the theory of slice regular (operatorial) functions.

\begin{Def}
Let $\A : D(\A) \funzione X$ be a closed right linear operator, let $\delta \in \cldxint{0,\pi/2}$ and let $\omega \in \erre$. We say that $\A$ is a \emph{spherical $\delta$-sectorial operator with vertex $\omega$} if 
\begin{equation*}
  \omega + \Omega_{\pi/2+\delta} = 
  \big\{q \in Q_\alg \setmeno \{\omega\} \, : \, \arg(q - \omega) < \pi/2 + \delta\big\} \subseteq \rho_\s(\A).
\end{equation*}
If $\A$ is a spherical $\delta$-sectorial operator with vertex $\omega$ for some $\delta \in \cldxint{0,\pi/2}$, then we say that $\A$ is a \emph{spherical sectorial operator with vertex $\omega$}. If in addition $\omega=0$, we simply say that $\A$ is a \emph{spherical sectorial operator}.
\end{Def}

The starting point of our analysis is the next result (cf. \cite[Theorem 5.6]{GhiRec15}) where we prove that
a spherical sectorial operator generates a strongly continuous semigroup represented by a suitable noncommutative
Cauchy integral formula. We state here this theorem in a form which is slightly more general than in \cite{GhiRec15}. In order to do this, we need some preparations.

Let $\j \in \su_\alg$. Recall that, given an interval $I$ of $\erre$, a $C^1$-path $\gamma:I \funzione \ci_\j$, a map 
$f:\gamma(I) \funzione X$ and a function $g:\gamma(I) \funzione \alg$, one can define the integral 
$\int_\gamma f(\alpha)\de \alpha \, g(\alpha) \in X$ (if it exists) by setting
\begin{equation}\label{def integrale}
\int_\gamma f(\alpha)\de \alpha \, g(\alpha) :=\int_If(\gamma(t))\gamma'(t)g(\gamma(t))\de t.
\end{equation}
If $\{\gamma_\ell:I_\ell \funzione \ci_\j\}_{\ell=1}^n$ is a finite family of $C^1$-paths of $\ci_\j$, $\Gamma$ is the formal sum $\sum_{\ell=1}^n\gamma_i$, and $f$ and $g$ are defined on $\bigcup_{\ell=1}^n\gamma_\ell(I_\ell)$, then we define $\int_\Gamma f(\alpha)\de \alpha \, g(\alpha):=\sum_{\ell=1}^n\int_{\gamma_\ell}f(\alpha)\de \alpha \, g(\alpha)$. If the image of $g$ is contained in $\ci_\j$, then we write $\int_\Gamma f(\alpha)g(\alpha)\de \alpha$ in place of $\int_\Gamma f(\alpha)\de \alpha \, g(\alpha)$, because $\gamma_\ell'(t)$ and $g(\gamma_\ell(t))$ commutes for every $\ell \in \{1,\ldots,n\}$ and for every $t \in I_\ell$. We refer the reader to \cite[Section 6]{GhiRec15} for more details concerning this kind of integrals.

Let $r \in \opint{0,\infty}$ and let $\eta \in \opint{0,\pi}$. Denote by $\ray^-(\j \, ;r;-\eta):\cldxint{-\infty,-r} \lra \CC_\j$, 
$\ce(\j \, ;r;\eta):\clint{-\eta,\eta} \funzione \ci_\j$ and $\ray^+(\j \, ;r;\eta):\cldxint{r,\infty} \lra \CC_\j$ the $C^1$-paths of 
$\ci_\j$ given by
\begin{alignat*}{4}
& \ray^-(\j \, ;r;-\eta)(t):=-te^{-\eta\j} & \qquad & \forall t \in \cldxint{-\infty,-r},\\
& \ce(\j \, ;r;\eta)(t):=re^{t\j} & \qquad & \forall t \in \clint{-\eta,\eta},\\
& \ray^+(\j \, ;r;\eta)(t):=te^{\eta\j}& \qquad & \forall t \in \clsxint{r,\infty}.
\end{alignat*}
Define $\Gamma(\j \, ;r;\eta)$ as the following formal sum of $C^1$-paths of $\ci_\j$:
\[
\Gamma(\j \, ;r;\eta):=\ray^-(\j \, ;r;-\eta)+\ce(\j \, ;r;\eta)+\ray^+(\j \, ;r;\eta).
\]

The mentioned 
slightly more general version of \cite[Theorem 5.6]{GhiRec15} reads as follows.

\begin{Thm}\label{main thm AMS}
Let $\A : D(\A) \funzione X$ be a spherical $\delta$-sectorial operator with vertex $\omega$. Suppose that $D(\A)$ is dense in $X$ and there exists $K > 0$ such that
\begin{equation*}
\norma{\C_q(\A)}{} \le \frac{K}{|q - \omega|} \qquad \forall q \in \omega + \Omega_{\pi/2+\delta}.
\end{equation*}
If $\j \in \su_\alg$, $r \in \opint{0,\infty}$ and $\eta \in \opint{\pi/2,\pi/2+\delta}$, then the integral
\begin{equation}
   \T(t) := \frac{1}{2\pi} \int_{\omega+\Gamma(\j \,;r;\eta)} \C_\alpha(\A) \, \j^{-1} e^{t\alpha}\de \alpha 
   \label{T(t), t > 0} \qquad \forall t > 0 
\end{equation}
is convergent in $\Lin^\r(X)$ and is independent of $\j$, $r$ and $\eta$. If we set $\T(0) := \Id$, then \eqref{T(t), t > 0} defines a strongly continuous semigroup $\T : \clsxint{0,\infty} \funzione \Lin^r(X)$ which is real analytic in $\opint{0,\infty}$ and whose generator is $\A$.
\end{Thm}

\begin{proof}
Here we show how to reduce to the case $\omega = 0$ which is dealt with in \cite[Theorem 5.6]{GhiRec15}. Define 
$\B := \A - \omega\Id$. By Lemma \ref{A <-> A - w}, we know  $\C_{q}(\A - \omega \Id) = \C_{q+\omega}(\A)$ for every 
$q \in \rho_\s(\A - \omega \Id)=\rho_\s(\A)-\omega$. In particular we get that $\B$ is spherical sectorial and 
$\norma{\C_q(\B)}{} \le K/|q|$ for every $q \in \Omega_{\pi/2+\delta}$, thus Theorem \cite[Theorem 5.6]{GhiRec15} applies and we get that $\B$ generates the strongly continuous semigroup $\S : \clsxint{0,\infty} \funzione \Lin^\r(X)$ defined by 
$\S(0) = \Id$ and $\S(t) := (1/2\pi) \int_{\Gamma(\j \, ;r,\eta)} \C_\alpha(\B) \, \j^{-1} e^{\alpha t} \de\alpha$ for $t > 0$, and real analytic in $\opint{0,\infty}$, where the integral is independent of $\j$, $r$ and $\eta$. Therefore, since 
$e^{\omega t} \in \erre$, we obtain that
$\clsxint{0,\infty} \funzione \Lin^\r(X) : t \longmapsto e^{\omega t} \S(t)$ is a strongly continuous semigroup generated by 
$\B+\omega\Id=\A$ and we have
\begin{align*}
  e^{\omega t} \S(t)
    & = e^{\omega t}\frac{1}{2\pi} \int_{\Gamma(\j \, ;r;\eta)} \C_\alpha(\B) \, \j^{-1} e^{\alpha t} \de\alpha 
    =  \frac{1}{2\pi} \int_{\omega+\Gamma(\j \, ;r;\eta)} \C_{\alpha}(\A) \, \j^{-1} e^{\alpha t} \de\alpha = \T(t)
\end{align*}
for all $t>0$. This completes the proof.
\end{proof}

In the next result we introduce a class of line integrals which extend \eqref{T(t), t > 0} to suitable spherical sectors
of $\alg$ and allow to infer the noncommutative semigroup law \eqref{eq:nsl}.

\begin{Lem}\label{T ind. of k}
Let $\A : D(\A) \funzione X$ be a spherical $\delta$-sectorial operator with vertex $\omega$ such that $D(\A)$ dense in $X$ and there exists $K > 0$ with
\begin{equation*}
\norma{\C_q(\A)}{} \le \frac{K}{|q - \omega|} \qquad \forall q \in \omega+\Omega_{\pi/2+\delta}.
\end{equation*}
If $\j \in \su_\alg$, $r \in \opint{0,\infty}$, $\eta \in \opint{\pi/2,\pi/2 + \delta}$ and $p, q \in \Omega_\delta$, then the integral
\begin{equation}
  \T_p(\j \, ;r;\eta,q) := 
  \frac{1}{2\pi} \int_{\omega+\Gamma(\j \,; r; \eta)} \left(\C_\alpha(\A) \, \j^{-1} \de\alpha\ \! \exp^\alpha_{p}(q)\right) \notag
\end{equation}
is absolutely convergent and defines a right slice regular function 
$\T_{p}(\j \, ;r;\eta,\cdot): \Omega_\delta \funzione \Lin^\r(X)$ such that 
\begin{equation}\label{uniform estimate for T(q)}
 \forall \delta' \in \opint{0,\delta} \quad  \exists M_{\delta'} \in \clsxint{1,\infty} \quad : \quad 
 \norma{\T_p(\j \, ;r;\eta,q)}{} \le M_{\delta'} e^{\omega \re(p+q)} \qquad \forall p, q \in \Omega_{\delta'},
\end{equation}
$(M_{\delta'}$ is independent of $q)$. Moreover $\T_p(\j \, ;r;\eta,q)$ does not depend on $\j \in \su_\alg$, $r \in \opint{0,\infty}$ and $\eta \in \opint{\pi/2,\pi/2 + \delta}$.
\end{Lem}

\begin{proof}
If $\alpha,p, q \in \ci_\j$ then points (i) and (ii) of Lemma \ref{L:the exponential exp_c(q)} imply that 
$\exp_p^\alpha(q) = \exp^\alpha(p+q) = e^{\alpha(p+q)} \in \ci_\j$. Thus, if $\gamma : \opint{0,1} \funzione \ci_\j$ is a piecewise $C^1$ reparametrization of $\Gamma(\j \,; r; \eta)$, we have 
\[
  \int_0^1 \left\|\C_{\omega+\gamma(t)}(\A) \, \j^{-1} \gamma'(t)\ \! \exp^{\omega+\gamma(t)}_{p}(q)\right\| \de t 
  \le e^{\omega \re(p+q)}\int_0^1 \frac{K}{|\gamma(t)|}|\gamma'(t)| e^{\re(\gamma(t)(p+q))} \de t,
\]
hence the absolute convergence of $\T_{p}(\j \, ;r;\eta,q)$ and \eqref{uniform estimate for T(q)} follow from the same argument of the complex case (see, e.g., \cite[Proposition 4.3, p. 97]{EngNag00}). Now we assume that 
$p = p_1 + p_2 \k \in \ci_\k$ and $q = q_1 + q_2 \mathbf{h} \in \ci_{\mathbf{h}}$ for some $\k, \mathbf{h} \in \su_\alg$, $p_m, q_m \in \erre$, $m = 1, 2$. Let us set $p_\j := p_1 + p_1\j \in \ci_\j$, $q_\j := q_1+ q_2\j \in \ci_\j$, and 
$q_\k := q_1 + q_2\k \in \ci_\k$. Since $\exp^\alpha_p$ is a right slice function, we can apply \eqref{repr. formula 1} twice and we find constants $b_m, c_m, d_m \in \alg$, $m = 1, 2$, such that
\begin{align}
  \exp_p^\alpha(q) 
    & =  \exp^\alpha_p(q_\k) b_1 +  \exp^\alpha_p(q_\k^c) b_2 \notag \\
    & =  \exp^\alpha(p + q_\k) b_1 + \exp^\alpha(p + q_\k^c) b_2 \notag \\
    & =  \exp^\alpha(p_\j + q_\j) c_1 + \exp^\alpha(p_\j^c + q_\j^c) d_1 + \exp^\alpha(p_\j + q_\j^c)  c_2 + \exp^\alpha(p_\j^c + q_\j)  d_2, \notag \\
    & =  e^{\alpha(p_\j + q_\j)} c_1 + e^{\alpha(p_\j^c + q_\j^c)} d_1 + e^{\alpha(p_\j + q_\j^c)} c_2 + e^{\alpha(p_\j^c + q_\j)}d_2. \notag
\end{align}
This formula allows to reduce to the previous case when $p, q \in \ci_\j$, thus the absolute convergence of
$\T_p(\j \, ;r;\eta,q)$ and estimate \eqref{uniform estimate for T(q)} are completely proved.

Fix $s \in \opint{0,\infty}$. Let us show that $\Omega_\delta \funzione \Lin^\r(X):p \longmapsto \T_p(\j \, ;r;\eta,s)$ is right slice regular. Consider the map $f_s:\opint{0,1} \times \Omega_\delta \funzione \Lin^\r(X)$ defined by
\[
f_s(t,p):=\C_{\omega+\gamma(t)}(\A) \, \j^{-1} \gamma'(t)\ \! \exp^{\omega+\gamma(t)}_{p}(s)=\big(\C_{\omega+\gamma(t)}(\A) \, \j^{-1} \gamma'(t)\big) \circ \big(\Id\exp^{\omega+\gamma(t)}_{p}(s)\big)
\]
By Lemma \ref{L:the exponential exp_c(q)}(iv), the function $p \longmapsto \exp^{\omega+\gamma(t)}_{p}(s)$ is right slice regular for every $t \in \opint{0,1}$. It follows immediately that, for every $t \in \opint{0,1}$, 
$p \longmapsto \Id\exp^{\omega+\gamma(t)}_{p}(s)$ is right slice regular and hence the same is true for $f_s(t,\cdot)$ (cf. Example \ref{exa:elementary}(b)). Thanks to Proposition \ref{slice regularity of integral functions}, we infer that 
$p \longmapsto \T_p(\j\, ;r;\eta,s)$ is right slice regular as well. Observe that
\begin{align}
  \T_t(\j \, ;r;\eta;s)=\T_0(\j \,;r;\eta,t+s)= \frac{1}{2\pi}\int_{\omega+\Gamma(\j \, ;r;\eta)} \C_\alpha(\A) \, \j^{-1} e^{\alpha(t+s)} \de\alpha \notag
    \qquad \forall t > 0.
\end{align}
Thanks to Theorem \ref{main thm AMS}, we know that $\T_0(\j \,;r;\eta,t+s)$ is independent of $\j$, $r$ and $\eta$. Therefore, by Lemma \ref{slice analytic continuation}, for every $s \in \opint{0,\infty}$ and for every $p \in \Omega_\delta$, we get that $\T_p(\j \, ;r;\eta,s)$ is independent of $\j$, $r$ and $\eta$. Now we fix $p \in \Omega_\delta$. Since 
$\exp_p^\alpha$ is right slice regular for every $\alpha$, proceeding as above, we obtain that 
$q \longmapsto \T_p(\j \,;r;\eta,q)$ is right slice regular. Furthermore, we proved that $\T_p(\j \, ;r;\eta,q)$ is independent of 
$\j$, $r$ and $\eta$ when $q \in \opint{0,\infty}$. Thus, by Lemma \ref{slice analytic continuation}, we get that 
$\T_p(\j \,;r;\eta,q)$ is independent of $\j$, $r$ and $\eta$ for every $p,q \in \Omega_\delta$.
\end{proof}

The previous Lemma \ref{T ind. of k} allows us to give the following definition.

\begin{Def}\label{Def Tp}
Let $\A : D(\A) \funzione X$ be a spherical $\delta$-sectorial operator of vertex $\omega$. Suppose that $D(\A)$ is dense in $X$ and there exists $K > 0$ such that
\begin{equation*}
\norma{\C_q(\A)}{} \le \frac{K}{|q - \omega|} \qquad \forall q \in \omega+\Omega_{\pi/2+\delta}.
\end{equation*}
For every $p \in \Omega_\delta \cup \{0\}$, we define $\T_{p} : \Omega_\delta \funzione \Lin^\r(X)$ by setting
\begin{equation*}
  \T_{p}(q) := 
  \frac{1}{2\pi} \int_{\omega+\Gamma(\j \, ;r;\eta)} \left(\C_\alpha(\A) \, \j^{-1} \de\alpha \ \! \exp^\alpha_{p}(q)\right), \qquad 
  q \in \Omega_\delta,
\end{equation*}
where $\j \in \su_\alg$, $r \in \opint{0,\infty}$ and $\eta \in \opint{\pi/2,\pi/2 + \delta}$ are arbitrarily chosen. Moreover we set $\T := \T_0$, i.e.
\begin{equation*}
  \T(q) := \T_0(q) = 
  \frac{1}{2\pi} \int_{\omega+\Gamma(\j \, ;r;\eta)} \left(\C_\alpha(\A) \, \j^{-1} \de\alpha \ \!\exp^\alpha(q) \right), \qquad 
  q \in \Omega_\delta.
\end{equation*}
\end{Def}

Since in general $\C_\alpha(\A)\j \neq \j\C_\alpha(\A)$ and $\C_\alpha(\A)\exp^\alpha(q) \neq \exp^\alpha(q) \C_\alpha(\A)$, the classical semigroup law fails for $\T$. In the following definition we introduce a new noncommutative semigroup law.

\begin{Def}
If $\delta \in \cldxint{0,\pi/2}$, then we say that $\T : \Omega_\delta \cup \{0\} \funzione \Lin^\r(X)$ is a 
\emph{right slice regular semigroup (of angle $\delta$)} if the restriction 
$\T|_{\Omega_\delta} : \Omega_\delta \funzione \Lin^\r(X)$ is right slice regular and
\begin{align}
  & \T(p+q) = \T(p) \odot_p \T(q) \qquad \text{$\forall p,q \in \Omega_\delta$ with $p+q \in \Omega_\delta$ and $pq = qp$}, \notag \\
  & \T(0) = \Id, \notag \\
  & \lim_{q \to 0} \T|_{\Omega_{\delta'}}(q)x = x \qquad \forall \delta' \in \opint{0,\delta}, \quad \forall x \in X. 
     \label{strong cont in 0}
\end{align}
If $\T$ is a slice regular semigroup of angle $\delta$ for some $\delta \in \cldxint{0,\pi/2}$, then we say that $\T$ is a \emph{right slice regular semigroup}. Moreover we say that a right slice regular semigroup $\T$ of angle $\delta$ is \emph{bounded} if
\begin{equation}
  \forall \delta' \in \opint{0,\delta} \quad \exists M_{\delta'} \in \clsxint{1,\infty} \quad  : \quad  
  \sup_{q \in \Omega_{\delta'}}\norma{\T(q)}{} \le M_{\delta'}. \notag
\end{equation}
\end{Def}


\begin{Lem}\label{analytic T -> strongly cont T}
If $\T : \Omega_\delta \cup \{0\} \funzione \Lin^\r(X)$ is a right slice regular semigroup, then its restriction 
$\T|_{\clsxint{0,\infty}}$ is a strongly continuous semigroup.
\end{Lem}

\begin{proof}
Let $\so{T} = (\so{T}_1, \so{T}_2) : D_\delta \funzione \Lin^\r(X) \otimes_\erre \ci$ be such that $\T = \II_\r(\so{T})$. Since $\T$ is right slice we have that $\so{T_2}(s) = 0$ for every $s>0$, thus we get, if $\j \in \su_\alg$,
\begin{align}
  \T(t+s) 
    & = \T(t) \odot_t \T(s) = (\T \odot \T(s))(t) = (\T \odot \T(s))(\phi_\j(t)) \notag \\
    & = \so{T_1}(t)\T(s) +  \so{T_2}(t)\T(s) \j = \so{T_1}(t)\T(s) \notag \\
    & = (\so{T_1}(t) + \so{T_2}(t)\j) \T(s) = \T(t)\T(s). \notag
\end{align}
The continuity of $\T|_{\clsxint{0,\infty}}(\cdot)x$ for $x \in X$ follows from \eqref{strong cont in 0}. 
\end{proof}

Thanks to the latter lemma, we can give the following definition.

\begin{Def}
Give a right slice regular semigroup $\T : \Omega_\delta \cup \{0\} \funzione \Lin^\r(X)$, we say that an operator 
$\A:D(\A) \funzione X$ is the \emph{generator of $\T$} if it is the generator of the strongly continuous semigroup 
$\T|_{\clsxint{0,\infty}}$. We say also that \emph{$\A$ generates $\T$}.
\end{Def}


\section{Spherical sectorial operators and slice regular semigroups} \label{S:sector <-> regular}

Throughout this section, we will assume that 
\[
\text{\emph{$\alg$ satisfies \eqref{assumption on A}, \eqref{S_A nonvuota} and \eqref{eq:assumption-2bis}, and $X$ is a Banach two-sided $\alg$-module.}}
\]
The main result of this paper reads as follows.

\begin{Thm}\label{thm:main}
The following assertions hold.
\begin{itemize}
\item[$(\mr{a})$]
Let $\A : D(\A) \funzione X$ be a spherical $\delta$-sectorial operator with vertex $\omega$. Suppose that $D(\A)$ is dense in $X$ and there exists $K > 0$ such that
\begin{equation*}
\norma{\C_q(\A)}{} \le \frac{K}{|q - \omega|} \qquad \forall q \in \omega+\Omega_{\pi/2+\delta}.
\end{equation*}
  Then $\A$ generates a right slice regular semigroup $\T : \Omega_\delta \cup \{0\} \funzione \Lin^\r(X)$ such that 
\begin{equation*} 
\forall \delta' \in \opint{0,\delta} \quad \exists M_{\delta'} \in \clsxint{1,\infty} \quad : \quad 
 \norma{\T(q)}{} \le M_{\delta'} e^{\omega \re(q)} \qquad \forall q \in \Omega_{\delta'}.
\end{equation*}
\item[$(\mr{b})$]
Let $\T : \Omega_\delta \cup \{0\} \funzione \Lin^\r(X)$ be a right slice regular semigroup and let 
$\A:D(\A) \funzione \Lin^\r(X)$ be its generator. Suppose there exist $\delta' \in \cldxint{0,\delta}$, $M \in \clsxint{1,\infty}$ and $\omega \in \erre$ such that $\norma{\T(q)}{} \le Me^{\omega \re(q)}$ for every $q \in \Omega_{\delta'}$. Then $\A$ is a spherical $\eta$-sectorial operator with vertex $\omega$ for some $\eta \in \cldxint{0,\pi/2}$. Moreover, there exists $K > 0$ such that
\begin{equation*}
\norma{\C_q(\A)}{} \le \frac{K}{|q - \omega|} \qquad \forall q \in \omega+\Omega_{\pi/2+\eta}.
\end{equation*}
\end{itemize}
\end{Thm}

Notice that bounded right slice semigroups are exactly those generated by spherical sectorial operators with 
vertex $\omega = 0$. We will give the proof of this result in the following two subsections.


\subsection{From spherical sectoriality to slice regularity}\label{S:sector -> regular}

Let us show that a spherical sectorial operator with vertex $\omega$ generates an exponentially bounded right slice regular semigroup.

\begin{Thm}\label{A sph sect -> T slic reg}
Let $\A : D(\A) \funzione X$ be a spherical $\delta$-sectorial operator with vertex $\omega$. Suppose that $D(\A)$ is dense in $X$ and there exists $K > 0$ such that
\begin{equation*}
\norma{\C_q(\A)}{} \le \frac{K}{|q - \omega|} \qquad \forall q \in \omega+\Omega_{\pi/2+\delta}.
\end{equation*}
Let $\T_{p} : \Omega_\delta \funzione \Lin^\r(X)$ be as in Definition \ref{Def Tp}. Then $\T_p$ is right slice regular for every $p \in \Omega_\delta \cup \{0\}$ and
\begin{equation*}
  \T_p(q) = \T(q) \odot_q \T(p) \qquad \forall p, q \in \Omega_\delta.
\end{equation*}

The function $\T = \T_0$ is the unique right slice regular function from $\Omega_\delta$ to $\Lin^\r(X)$ which coincides with the semigroup generated by $\A$ on $\opint{0,\infty}$. Moreover, it holds
\begin{equation}\label{semigroup property for T(p)}
  \T(p+q) = \T(p) \odot_p \T(q) = \T(q) \odot_q \T(p) \quad \text{$\forall p,q \in \Omega_\delta$ with $p+q \in \Omega_\delta$ and $pq = qp$}, \notag 
\end{equation}
i.e.
\begin{equation}
  \T(p+q) = (\T \odot \T(q))(p) = (\T \odot \T(p))(q) \quad \text{$\forall p,q \in \Omega_\delta$ with $p+q \in \Omega_\delta$ and $pq = qp$}, \notag
\end{equation}
where in the second term $\T(q)$ is the constant function $p \longmapsto \T(q)$ and in the last term $\T(p)$ is the constant function $q \longmapsto \T(p)$.
\end{Thm}

\begin{proof}
We already proved in Lemma \ref{T ind. of k} that $\T_p$ is a right slice regular function. Let us explicitly write its stem function in the particular case $p = 0$. Given $\alpha \in \alg$, we define $E^\alpha=(E^\alpha_1,E^\alpha_2):D_\delta \funzione \Lin^\r(X) \otimes_\erre \ci$ by setting $E^\alpha_1(z):=\sum_{n \geq 0}\frac{\alpha^n}{n!}\Re(z^n)$ and $E^\alpha_2(z):=\sum_{n \geq 0}\frac{\alpha^n}{n!}\Im(z^n)$. It follows immediately that 
$\exp_0^\alpha = \exp^\alpha = \II_\r(E^\alpha)$. Then, from Proposition \ref{slice regularity of integral functions} and Example \ref{exa:elementary}(b), we obtain that $\T = \T_0 = \II_\r(\so{T})$, where $\so{T} = (\so{T}_1, \so{T}_2)$ and
\begin{equation}
  \so{T}_m(z) := \frac{1}{2\pi} \int_{\omega+\Gamma(\j \, ;r,\eta)} \left(\C_\alpha(\A) \, \j^{-1} \de\alpha \ \!E^\alpha_m(z)\right) \qquad \forall z \in D_\delta, \quad  m = 1, 2 \notag
\end{equation}
for some fixed $\j \in \su_\alg$, $r \opint{0,\infty}$ and $\eta \in \opint{\pi/2,\pi/2+\delta}$. Thanks to Lemma 
\ref{L:the exponential exp_c(q)}(ii), we have
\[
  \T(t) = \frac{1}{2\pi} \int_{\omega+\Gamma(\j \, ;r;\eta)} \C_\alpha(\A) \, \j^{-1} e^{\alpha t} \de\alpha \qquad \forall t > 0,
\]
therefore from Theorem \ref{main thm AMS} 
it follows that
\begin{equation}\label{T(t) semigruppo}
  \T(t + s) = \T(t)\T(s) = \T(s)\T(t) \qquad \forall t, s > 0.
\end{equation}
Now for any $t > 0$ let us consider the mappings
\begin{equation} 
  \T_{t} = \T(t + \cdot) : \Omega_\delta \funzione \Lin^\r(X) : p \longmapsto \T(t+p) \notag
\end{equation}
(cf. formula \eqref{exp_q(p) if pq=qp}) and
\begin{equation}
  \U_{t} := \T(t)\T(\cdot) : \Omega_\delta \funzione \Lin^\r(X): p \longmapsto \T(t)\T(p). \notag
\end{equation}
Clearly $\U_t$ is right slice regular (cf. Example \ref{exa:elementary}(b)) and from \eqref{T(t) semigruppo} it follows that 
$\T_t(s) = \U_t(s)$ for every $s > 0$. Therefore, since $\T_t$ is right slice regular 
by Lemma \ref{T ind. of k}, we obtain that $\T_t = \U_t$, hence
\begin{equation*}\label{T(p+t) = T(t) T(p)}
  \T(p+t) = \T(t) \T(p) \qquad \forall p \in \Omega_\delta, \ \forall t > 0.
\end{equation*}
Now we fix $p$ and consider the slice right regular function $\T_p : \Omega_\delta \funzione \Lin^\r(X)$. If 
$p,q \in \Omega_\delta$ with $p+q \in \Omega_\delta$ and $pq = qp$, then Lemma \ref{L:the exponential exp_c(q)}(i) implies that
\begin{align}
  \T_p(q) 
    & = \frac{1}{2\pi} \int_{\omega+\Gamma(\j \, ;r;\eta)} \left(\C_\alpha(\A) \, \j^{-1} \de\alpha\ \! \exp^\alpha_p(q)\right) \notag \\
    & = \frac{1}{2\pi} \int_{\omega+\Gamma(\j \, ;r;\eta)} \left(\C_\alpha(\A) \, \j^{-1} \de\alpha\ \! \exp^\alpha(p+q) \right)
    = \T(p+q), \label{p+q}
\end{align}
in particular $\T_p(t) = \T(p + t)$ for every $t > 0$. Thus from \eqref{T(p+t) = T(t) T(p)} we obtain
\begin{equation*}
  \T_p(t) = \T(t)\T(p) \qquad \forall p \in \Omega_\delta, \ \forall t > 0.
\end{equation*}
For every $p \in \Omega_\delta$, define the right slice regular function $\V_p : \Omega_\delta \funzione \Lin^\r(X)$ by setting
\begin{equation*}
  \V_p := \T \odot \T(p)
\end{equation*}
(according to our notation $\T(p)$ is here the constant function $q \longmapsto \T(p)$). If $z \in D_\delta$ and 
$q:=\phi_\j(z)$, then we have
\begin{align}
  E^\alpha_1(z)  \T(p) + E^\alpha_2(z)  \T(p) \j
    & = \sum_{n \ge 0}\frac{\alpha^n}{n!} \Re(z^n) \T(p) + \sum_{n \ge 0}\frac{\alpha^n}{n!} \Im(z^n)  \T(p) \j \notag \\
    & = \sum_{n \ge 0}\frac{\alpha^n}{n!} \T(p)  \Re(z^n) + \sum_{n \ge 0}\frac{\alpha^n}{n!}   \T(p) \Im(z^n)  \j \notag \\
    & = \sum_{n \ge 0}\frac{\alpha^n}{n!} \T(p)  (\Re(z^n) + \Im(z^n)  \j) 
       = \sum_{n \ge 0}\frac{\alpha^n}{n!} \T(p) q^n, \notag
\end{align}
hence
\begin{align}
  (\T \odot \T(p))(q) 
    & = \so{T}_1(z) \T(p)  + \so{T}_2(z) \T(p)\j \notag \\
    & = \frac{1}{2\pi} \int_{\omega+\Gamma(\j \, ;r;\eta)} \left(\C_\alpha(\A) \, \j^{-1} \de\alpha \ \!E^\alpha_1(z) \T(p)\right) \notag \\
    & \sp + \frac{1}{2\pi} \int_{\omega+\Gamma(\j \, ;r;\eta)} \C_\alpha(\A) \, \j^{-1} \de\alpha \ \!E^\alpha_2(z)  \T(p) \j \notag \\
    & = \frac{1}{2\pi} 
           \int_{\omega+\Gamma(\j \, ;r;\eta)} 
              \left(\C_\alpha(\A) \, \j^{-1} \de\alpha\  \!(E^\alpha_1(z)  \T(p) + E^\alpha_2(z)  \T(p)\j)\right)
           \notag \\
    & = \frac{1}{2\pi} \int_{\omega+\Gamma(\j \, ;r;\eta)} \Big(\C_\alpha(\A) \, \j^{-1} \de\alpha \ \sum_{n \ge 0}\frac{\alpha^n}{n!}
          \T(p)q^n \Big) .\notag      
\end{align}
Therefore if $t > 0$ we get
\begin{align}
  (\T \odot \T(p))(t)
    & = \frac{1}{2\pi}
      \int_{\omega+\Gamma(\j \, ;r;\eta)} \Big(\C_\alpha(\A) \, \j^{-1} \de\alpha \ \!\sum_{n \ge 0}\frac{\alpha^n}{n!} \T(p) t^n \Big)
      \notag \\
    & = \frac{1}{2\pi} 
       \int_{\omega+\Gamma(\j \, ;r;\eta)}  
         \Big( \C_\alpha(\A) \, \j^{-1} \de\alpha\ \! \sum_{n \ge 0}\frac{\alpha^n}{n!}t^n \T(p) \Big) \notag \\
    & = \frac{1}{2\pi} 
       \int_{\omega+\Gamma(\j \, ;r;\eta)}  \left( \C_\alpha(\A) \, \j^{-1} \de\alpha\ \! \exp^\alpha(t)\right) \T(p) \notag \\   
        & =    \T(t) \T(p), \notag
\end{align}
thus
\begin{equation*}
    \V_p(t) = \T_p(t)  \qquad \forall t > 0.
\end{equation*}
Since $\V_p$ and $\T_p$ are both right slice regular, Lemma \ref{slice analytic continuation} yields
\begin{equation*}
  \T_p(q) = (\T \odot \T(p))(q) \qquad \forall p, q \in \Omega_\delta
\end{equation*}
and, by virtue of \eqref{p+q},
\begin{equation*}
 \T(p+q) = (\T \odot \T(p))(q) \qquad \text{$\forall p,q \in \Omega_\delta$ with $p+q \in \Omega_\delta$ and $pq = qp$}.
\end{equation*}
This completes the proof.
\end{proof}


\subsection{From slice regularity to spherical sectoriality}\label{S:regular -> sector}

Our final task is proving that a right slice regular semigroup is generated by a spherical sectorial operator.
We need the following lemma providing the estimate \eqref{stima risolvente analitico}
for the spherical resolvent operator of the generator of a semigroup. As in the classical case, the Laplace transform \eqref{Res = Lap} is a crucial tool for this proof, but in our noncommutative framework things are complicated by the fact that $q \longmapsto (\T(t)e^{-tq})x$ is not right slice regular even if $q \longmapsto \T(t)e^{-tq}$ is (cf. Remark 
\ref{eq:no-right-slice}).

\begin{Lem}
Let $\delta \in \opint{0,\pi/2}$ and let $\T : \Omega_\delta \cup \{0\} \funzione \Lin^\r(X)$ be a function such that 
$\T|_{\Omega_\delta}$ is right slice regular and $\T|_{\clsxint{0,\infty}}$ is a strongly continuous semigroup. Suppose there exist $\delta' \in \cldxint{0,\delta}$, $M \in \clsxint{1,\infty}$ and $\omega \in \erre$ such that 
$\norma{\T(q)}{} \le Me^{\omega \re(q)}$ for every $q \in \Omega_{\delta'}$. If $\A$ is the generator of $\T$, then 
$\rho_\s(\A) \subseteq \omega+\Omega_{\pi/2}$ and there exists $K > 0$ such that 
\begin{equation}\label{stima risolvente analitico}
  \norma{\C_q(\A)}{} \le \frac{K}{|q - \omega|} \qquad \forall q \in \omega + \Omega_{\pi/2}.
\end{equation}
\end{Lem}

\begin{proof}
Thanks to Theorem \ref{re q > 0 in rho(A)}, we know that $\rho_\s(\A) \subseteq \omega+\Omega_{\pi/2}$ and \eqref{stima risolvente analitico} is true for $q \in \opint{\omega,\infty}$. Fix $q \in (\omega + \Omega_{\pi/2}) \setmeno \erre$. Let $\j \in \su_\alg$ and $r, s \in \erre$ be such that $s > 0$ and $q = r+s\j$. Define $\lambda \in D_\delta$ by $\lambda:=r+is$, thus $q=\phi_\j(\lambda)$. The function $\T_\j : D_\eta \funzione (\Lin^\r(X))_\j$ is holomorphic by Proposition \ref{f regular iff f slice holomorphic}, hence $\T|_{\clsxint{\eps,\infty}} \in C(\clsxint{\eps,\infty};\Lin^\r(X))$ for every $\eps > 0$. Using again Theorem \ref{re q > 0 in rho(A)}, we infer that
\begin{align}\label{eq 3}
  \C_q(\A)x 
     & = \int_0^\eps (\T(t)e^{-tq})x \de t  + \int_\eps^\infty (\T(t)e^{-tq})x \de t\notag \\
     & = \int_0^\eps (\T(t)e^{-tq})x \de t + \left(\int_\eps^\infty \T(t)e^{-tq} \de t\right) x \notag \\
     & = \int_0^\eps (\T(t)e^{-tq})x \de t +  \left(\int_\eps^\infty e^{-t \lambda}\T_\j(t) \de t\right) x.
\end{align}
The mapping $z \longmapsto e^{-z\lambda}\T_\j(z)$ is holomorphic from $D_\delta$ into $(\Lin^\r(X))_\j$, hence a standard argument allows us to change the path of integration in the last integral in \eqref{eq 3} from 
$\clsxint{\eps, \infty}$ into $\gamma_\eps : \clsxint{0,\infty} \funzione \ci : \rho \longmapsto \eps + \rho e^{\theta i}$ where $\theta$ is a fixed element of $\opint{-\delta',0}$. Therefore we get
\begin{align}
 \int_\eps^\infty e^{-t \lambda}\T_\j(t) \de t = \int_{\gamma_\eps} e^{-z\lambda} \T_\j(z) \de z 
  = \int_0^\infty e^{\theta i}  e^{-\lambda(\eps+\rho e^{\theta i})} \T_\j(\eps + \rho e^{\theta i})  \de\rho, \notag
\end{align}
thus, if $C := -M/\sin\theta$, 
then the following standard estimate can be obtained on the interval $\clsxint{\eps,\infty}$ (rather than in 
$\clsxint{0,\infty}$):
\begin{align}
  \left\| \int_\eps^\infty e^{-t \lambda}\T_\j(t) \de t \right\|
    & \le \int_0^\infty 
            Me^{-\eps(\re(\lambda) - \omega)} e^{-\re((\lambda-\omega)\rho e^{\theta i})} \de \rho \notag \\
    & \le \int_0^\infty M e^{-\rho((r-\omega)\cos\theta - s\sin\theta)}\de\rho \notag \\
    & = \frac{M}{(r-\omega)\cos\theta - s\sin\theta} \le \frac{M}{ - s\sin\theta} = \frac{C}{s}=\frac{C}{|\im(q)|}. \notag
\end{align}
In this way, for every $x \in X$, we find
\begin{align}
  \norma{\C_{q}(\A)x}{} 
     & \le \int_0^\eps \norma{(\T(t)e^{-tq})x}{} \de t + \frac{C}{|\im(q)|}\norma{x}{} \notag \\
     & \le M \int_0^\eps e^{t(\omega-\re(q))}\norma{x}{} \de t+ \frac{C}{|\im(q)|}\norma{x}{}= 
     \left(M \frac{1 -e^{\eps(\omega-\re(q))}}{\re(q)-\omega} + \frac{C}{|\im(q)|}\right) \norma{x}{}, \notag 
\end{align}
hence, by the arbitrariness of $\eps$, we obtain
\begin{align}\label{stima risolvente con im(q)}
  \norma{\C_{q}(\A)}{} \le \frac{C}{|\im(q)|}. 
\end{align}
Collecting together \eqref{stima risolvente con re(q)} with \eqref{stima risolvente con im(q)} we find
\eqref{stima risolvente analitico}.
\end{proof}

\begin{Thm}\label{slice regularity+bddness -> sph. sectoriality}
Let $\A : D(\A) \funzione X$ be a closed right linear operator with $D(\A)$ dense in $X$. Suppose that $\A$ generates a strongly continuous semigroup $\T:\clsxint{0,\infty} \funzione \Lin^\r(X)$ and there exist $M \in \clsxint{1,\infty}$, 
$\omega \in \erre$ and $L > 0$ such that $\norma{\T(t)}{} \le M e^{\omega t}$ for all $t \ge 0$, and
\begin{equation}\label{sectorial estimate2}
  \norma{\C_q(\A)}{} \le \frac{L}{|\im(q)|} \qquad \forall q \in \omega + \Omega_{\pi/2}.
\end{equation}
Then $\A$ is a spherical $\delta$-sectorial operator with vertex $\omega$ for some $\delta \in \cldxint{0,\pi/2}$. Moreover, there exists $K > 0$ such that
\[
\norma{\C_q(\A)}{} \le \frac{K}{|q-\omega|} \qquad \forall q \in \omega + \Omega_{\pi/2+\delta}.
\]
\end{Thm}

\begin{proof}
If $\B := \A - \omega\Id$, then $\B$ generates the strongly continuous semigroup $\S(t) = e^{-\omega t}\T(t)$ statisfying the estimate $\norma{\S(t)}{} \le M$. Therefore $\Omega_{\pi/2} \subseteq \rho_\s(\B)$ and \eqref{sectorial estimate2} yields
\begin{equation}\label{C(B) <}
  \norma{\C_{q}(\B)}{} = \norma{\C_{q+\omega}(\A)}{} \le \frac{L}{|\im(q)|} \qquad \forall q \in \Omega_{\pi/2}.
\end{equation}
Moreover, thanks to \eqref{stima risolvente con re(q)}, for every $q \in \Omega_{\pi/2}$ we have 
$\norma{\C_q(\B)}{} \le M/\re(q)$, which together with \eqref{C(B) <} yields 
\begin{equation} \label{cq}
  \norma{\C_{q}(\B)}{} \le \frac{C}{|q|} \qquad \forall q \in \Omega_{\pi/2}
\end{equation}
for some $C > 0$. Fix $p \in \Omega_{\pi/2}$. By definition of spherical resolvent operator we have 
$\C_p(\B) = \Q_p(\B)p^c - \B\Q_p(\B)$ and $\C_{p^c}(\B) = \Q_{p^c}(\B)p - \B\Q_{p^c}(\B) =\Q_{p^c}(\B)p - \B\Q_p(\B)$. Hence, subtracting the two identities, we get $\C_{p^c}(\B) - \C_p(\B)  = \Q_p(\B)(p - p^c) = \Q_p(\B)2 \im(p)$. Therefore if $\im(p) \neq 0$ then
\begin{equation*}
  \Q_p(\B) = (\C_{p^c}(\B) - \C_p(\B)) (2\im(p))^{-1},
\end{equation*}
therefore 
it follows that 
\begin{equation}
  \norma{\Q_p(\B)}{}
    \le \frac{1}{2|\im(p)|}(\norma{\C_{p^c}(\B)}{} + \norma{\C_p(\B)}{}) 
    \le \frac{C}{|\im(p)||p|}. \label{stima}
\end{equation}
Now fix $x \in X$, $\j \in \su_\alg$ and let $\mu \in \ci$ be such that $\phi_\j(\mu) = p$. From \eqref{resolvent operator} and \eqref{Q = R R}, we infer the following chain of equalities:
\begin{align}
  \R_\mu(\B_\j)x 
    & = ({\overline{\mu}} \Id_{X_\j} - \B_\j)\Q_p(\B)x \notag \\
    & = {\overline{\mu}} \Q_p(\B)x - \B_\j\Q_p(\B)x \notag \\
    & = (\Q_p(\B)x) p^c - \B(\Q_p(\B)x) \notag \\
    & = \Q_p(\B)(x p^c) + \C_p(\B)x - \Q_p(\B)(p^c x), \notag
\end{align}
therefore from \eqref{cq} and \eqref{stima} we get
\begin{align}
  \norma{\R_{\mu}(\B_\j)x}{} 
    & \le \norma{\Q_p(\B)(x p^c)}{} + \norma{\C_p(\B)x}{} + \norma{\Q_p(\B)(p^c x)}{}  \notag \\
    & \le  \norma{\Q_p(\B)}{}\norma{xp^c}{} + \norma{\C_p(\B)}{}\norma{x}{} + \norma{\Q_p(\B)}{}\norma{p^c x}{}  \notag \\
    & =  \norma{\Q_p(\B)}{}\norma{x}{}|p| + \norma{\C_p(\B)}{}\norma{x}{} + \norma{\Q_p(\B)}{}\norma{x}{}|p|  \notag \\
    & \le \left(\frac{C|p|}{|\im(p)||p|} + \frac{C}{|p|}+ \frac{C|p|}{|\im(p)||p|}\right)\norma{x}{} \le
    \frac{3C}{|\im(p)|}\norma{x}{} = \frac{3C}{|\im(\mu)|}\norma{x}{}. \notag
\end{align}
Hence, by the arbitrariness of $p \in \Omega_{\pi/2}$ and of $x \in X$, we have proved that
\[
  \norma{\R_\mu(\B_\j)}{} \le \frac{3C}{|\im(\mu)|} \qquad \forall \mu \in D_{\pi/2}.
\]
Therefore the classical complex theory of analytic semigroups (see, e.g., \cite[Thereom 4.6]{EngNag00}) applies to $\B_\j$ and we find that there exists $\delta_\j \in \cldxint{0,\pi/2}$ such that $D_{\pi/2 + \delta_\j} \subseteq \rho(\B_\j)$,
the resolvent set of $\B_\j$. This fact allows to apply Theorem \ref{legame risolventi} and to deduce that $\B$ is spherical sectorial of angle $\delta:=\sup_{\j \in \su_\alg}\delta_\j$, i.e. 
$\Omega_{\pi/2 + \delta} \subseteq \rho_\s(\B)$. Moreover, we have that $D_{\pi/2+\delta} \subseteq \rho(\B_\j)$ for all 
$\j \in \su_\alg$ (in other words, we can assume that $\delta_\j$ does not depend on $\j \in \su_\alg$).
Now let $q \in \Omega_{\pi/2 + \delta}$ and let $\lambda \in D_{\pi/2 + \delta}$ be such that $q = \phi_\j(\lambda)$. From the classical theory we also have that there is a constant $N > 0$ such that 
$\norma{\R_\lambda(\B_\j)}{} \le N/|\lambda|=N/|q|$, hence \eqref{Q = R R} yields
\begin{equation*}
  \norma{\Q_{q}(\B)}{} \le \norma{\R_{\overline{\lambda}}(\B_\j)}{}\norma{\R_\lambda(\B_\j)}{} \le 
  \frac{N}{|\lambda|^2} = \frac{N}{|q|^2},
\end{equation*}
therefore, observing that $\B\R_z(\B_\j) = \B_\j\R_z(\B_\j) = z\R_z(\B_\j) - \Id_{X_\j}$ for every $z \in \rho(\B_\j)$, we get
\begin{align*}
  \norma{\C_q(\B)}{} 
    & = \norma{\Q_q(\B)q^c - \B\Q_q(\B)}{} \le 
    \norma{\Q_q(\B)q^c}{} + \norma{\B\Q_q(\B)}{} \notag \\
    & = \norma{\Q_q(\B)q^c}{} + \norma{\B\R_{\overline{\lambda}}(\B_\j)\R_\lambda(\B_\j)}{} \notag \\
    & = \norma{\Q_q(\B)}{}|q^c| + \norma{(\overline{\lambda}\R_{\overline{\lambda}}(\B_\j) - \Id_{X_\j})\R_\lambda(\B_\j)}{} 
          \notag \\
    & \le \norma{\Q_q(\B)}{}|q^c| + |\lambda|\norma{\R_{\overline{\lambda}}(\B_\j)\R_\lambda(\B_\j)}{} +
      \norma{\R_\lambda(\B_\j)}{} \notag \\
    & \le \frac{N|q^c|}{|q|^2} + \frac{N|q|}{|q^2|} + \frac{N}{|q|} = \frac{3N}{|q|}. 
    \end{align*}
Now we conclude by invoking the equalities $\rho_\s(\A) = \omega + \rho_\s(\B)$ and $\C_q(\A) = \C_{q-\omega}(\B)$.
\end{proof}


\vspace{1em}

\noindent \textbf{Acknowledgements.}  The first author is partially supported by INFN-TIFPA and by GNSAGA of INdAM. The second author is partially supported by GNAMPA of INdAM. 




\begin{thebibliography}{40}

\bibitem{Adl95}
  S. Adler, 
  ``Quaternionic Quantum Field Theory'', 
  Oxford University Press, 1995.

\bibitem{ACKS-in-press}
D. Alpay, F. Colombo, D.P. Kimsey, I. Sabadini, \emph{The Spectral Theorem for Unitary Operators Based on the $S$-Spectrum}, Milan Journal of Mathematics, in press.
DOI:10.1007/s00032-015-0249-7

\bibitem{ACK16}
D. Alpay, F. Colombo, D.P. Kimsey, \emph{The spectral theorem for quaternionic unbounded normal operators based on the S-spectrum}, J. Math. Phys. {\bf 57}, 023503, 27 pp. (2016).

\bibitem{AlpColLewSab15}
  D. Alpay, F. Colombo, I. Lewkowicz, I. Sabadini, 
  \emph{Realizations of Slice Hyperholomorphic Generalized Contractive and Positive Functions}, 
  Milan J. Math., \textbf{83} (2015), 91--144.      

\bibitem{AndFul74}
  F.W. Anderson, K.R. Fuller, 
  ``Rings and Categories of Modules'', 
  Springer Verlag, New York, 1974. 

\bibitem{BirNeu36}
  G. Birkhoff, J. von Neumann, 
  \emph{The logic of quantum mechanics}, 
  Ann. of Math. (2), \textbf{37} (1936), 823--843. 

\bibitem{CasTru85}
G. Cassinelli, P. Truini, \emph{Quantum mechanics of the quaternionic Hilbert spaces based upon the imprimitivity theorem}, Rep. Math. Phys. \textbf{21} (1985), no. 1, 43--64.

\bibitem{CoGeSa}
  F. Colombo, G. Gentili, I. Sabadini, 
  \emph{A Cauchy kernel for slice regular functions}, 
  Ann. Glob. Anal. Geom., \textbf{37} (2010), 361--378.

\bibitem{CoGeSaSt07}
F. Colombo, G. Gentili, I. Sabadini, D.C. Struppa, \emph{A functional calculus in a non commutative
setting}. Electron. Res. Announc. Math. Sci., \textbf{14} (2007), 60--68.

\bibitem{CoGeSaSt}
  F. Colombo, G. Gentili, I. Sabadini, D.C. Struppa, 
  \emph{Non Commutative Functional Calculus: Bounded Operators}, Complex Anal. Oper. Theory., \textbf{4} (2010), 821--843. 

\bibitem{CoGeSaSt10}
  F. Colombo, G. Gentili, I. Sabadini, D.C. Struppa, 
  \emph{Non-commutative functional calculus: Unbounded operators}, 
  J. Geom. Phys., \textbf{60} (2010), 251--259.       

\bibitem{ColSab09}
  F. Colombo, I. Sabadini, 
  \emph{On some properties of the quaternionic functional calculus}, 
  J. Geom. Anal., \textbf{19} (2009), 601--627.  

\bibitem{ColSab10}
  F. Colombo, I. Sabadini, 
  \emph{On the formulation of the quaternionic functional calculus}, 
  J. Geom. Phys., \textbf{60} (2009), 1490--1508.    

\bibitem{ColSab11}
  F. Colombo, I. Sabadini, 
  \emph{The quaternionic evolution operator}, 
  Adv. Math., \textbf{227} (2011), 1772--1805.

\bibitem{ColSab14}
  F. Colombo, I. Sabadini, 
  \emph{The F -functional calculus for unbounded operators}, 
  J. Geom. Phys., \textbf{86} (2014), 392--407.    
  
\bibitem{ColSabStr08}
  F. Colombo, I. Sabadini, D.C. Struppa, 
  \emph{A new functional calculus for noncommuting operators}, 
 J. Functional Analysis, \textbf{254} (2008), 2255--2274.

\bibitem{CoSaSt}
  F. Colombo, I. Sabadini, D.C. Struppa, 
  ``Noncommutative Functional Calculus'', 
  Birkhauser, Basel, 2011.

\bibitem{Dav80}
  E.B. Davies, 
  ``One Parameter Semigroups'', 
  Academic Press, 1980.

\bibitem{numbers}
   H.-D. Ebbinghaus, H. Hermes, F. Hirzebruch, M. Koecher, K. Mainzer, J. Neukirch, A. Prestel, R. Remmert,
   ``Numbers'', Grad. Texts in Math., vol. 123, Springer-Verlag, New York, 1990.

\bibitem{Emc63}
  G. Emch, 
  \emph{M\'ecanique quantique quaternionienne et relativit\'e restreinte}, 
  Elv. Phys. Acta, \textbf{36} (1963), 739--769.    

\bibitem{EngNag00}
  K.-J. Engel, R. Nagel, 
  ``One-Parameter Semigroups for Linear Evolution Equations'', 
  Springer Verlag, New York, 2000.

\bibitem{FinJauSchSpe62}
  D. Finkelstein, J.M. Jauch, S. Schiminovich, D. Speiser, 
  \emph{Foundations of quaternionic quantum mechanics}, 
  J. Mathematical Phys., \textbf{3} (1962), 207--220.

\bibitem{GenSto12}
  G. Gentili, C. Stoppato, 
  \emph{Power series and analyticity over the quaternions}, 
  Math. Ann., \textbf{352} (2012), 113-131.

\bibitem{GeSt}
  G. Gentili, D.C. Struppa, 
  \emph{A new theory of regular functions of a quaternionic variable}, 
  Adv. Math., \textbf{216} (2007), 279--301.

\bibitem{GhiMorPer13}
  R. Ghiloni, V. Moretti, A. Perotti,
  \emph{Continuous slice functional calculus in quaternionic Hilbert spaces}, 
  Rev. Math. Phys. \textbf{25} (2013), no. 4, 1330006, 83 pp. 

\bibitem{GhiMorPer-bis}
R. Ghiloni, V. Moretti, A. Perotti,
  \emph{Spectral representations of normal operators in quaternionic Hilbert spaces via intertwining quaternionic PVMs}, arXiv:1602.02661v1

\bibitem{GhiPer11}
  R. Ghiloni, A. Perotti, 
  \emph{Slice regular functions on real alternative algebras}, Adv. Math., \textbf{226} (2011), 1662--1691.

\bibitem{GhiPer14}
  R. Ghiloni, A. Perotti,
  \emph{Power and spherical series over real alternative {$^*$}-algebras}, Indiana Univ. Math. J. \textbf{63}~(2) (2014), 495--532.

\bibitem{GhPeSt} R. Ghiloni, A. Perotti, C. Stoppato, \emph{The algebra of slice function}, to appear in Trans. Amer. Math. Soc.

\bibitem{GhiRec15}
R. Ghiloni, V. Recupero, 
  \emph{Semigroups over real alternative *-algebras: Generation theorems and spherical sectorial operators}, Trans. Amer. Math. Soc., \textbf{368} (2016), no. 4, 2645--2678.

\bibitem{GM91} J.E. Gibert, M.A.M Murray, ``Clifford algebras and Dirac operators in harmonic analysis'', Cambridge University Press, Cambridge, 1991.

\bibitem{Gol85}
  J.A. Goldstein, 
  ``Semigroups of Operators and Applications'', 
  Oxford University Press, 1985.

\bibitem{GHS08}
K. G{\"u}rlebeck, K. Habetha, W. Spr{\"o}{\ss}ig, ``Holomorphic functions in the plane and $n$-dimensional space'', Birkh\"auser Verlag, Basel, xiv+394, 2008.

\bibitem{HilPhi57}
  E. Hille, R.S. Phillips, 
  ``Functional Analysis and Semigroups'',
   Amer. Math. Soc. Coll. Publ., vol. 31, Amer. Math. Soc., 1957.

\bibitem{HorBie84}
  L. P. Horwitz, L.C. Biedenharn, 
  \emph{Quaternionic quantum mechanics: Second quantization and gauge field}, 
  Annals of Physics, \textbf{157} (1984), 432--488.

\bibitem{Kap}
I. Kaplansky, 
\emph{Normed algebras}, Duke Mat. J., \textbf{16} (1949), 399--418.    

\bibitem{Lam91}
  T.Y. Lam,
  ``A First Course in Noncommutative Rings'', 
  Grad. Texts in Math., Vol. 131, Springer-Verlag, New York, 1991.

\bibitem{Lun95}
  A. Lunardi, 
  ``Analytic Semigroups and Optimal Regularity in Parabolic Problems'', 
  Birkhauser-Verlag, 1995.

\bibitem{Neu32}
  J. von Neumann, 
  \emph{uber einen Satz von Herrn M.H. Stone}, 
  Ann. of Math. \textbf{33} (1932), 567--573.
  
\bibitem{Neu32b}
  J. von Neumann, 
  ``Mathematische Grundlagen der Quantenmechanik'',
  Springer-Verlag, 1932.
  
\bibitem{Paz83}
  A. Pazy, 
  ``Semigroups of Linear Operators and Applications to Partial Differential Equations'', 
  Springer-Verlag, New York, 1983.

\bibitem{Rud73}  
  W. Rudin, 
  ``Functional Analysis'', 
  McGraw-Hill, New York, 1977.

\bibitem{Sol95}
M.P. Sol\`er,
\emph{Characterization of Hilbert spaces by orthomodular spaces}, Comm. Algebra \textbf{23} (1995), no. 1, 219--243.

\bibitem{Sto32}
  M.H. Stone, 
  \emph{On one-parameter unitary groups in Hilbert spaces}, Ann. of Math. \textbf{33} (1932), 643--648.

\bibitem{Tai95}
K. Taira, ``Analytic Semigroups and Semilinear Initial Boundary Value Problems'', London Math. Soc. Lect. Notes Ser., vol. 223, Cambridge University Press, 1995.
\end{thebibliography}
\end{document}